\documentclass[11pt,reqno]{amsart}

\usepackage{amsthm,amsmath,amssymb}
\usepackage{graphicx}
\usepackage{float}
\usepackage[colorlinks=true,
linkcolor=blue,
urlcolor=blue,
citecolor=blue]{hyperref}
\usepackage{mathtools}
\usepackage{relsize}
\usepackage{tikz}
 \usepackage{enumerate}
\usepackage[toc,page]{appendix}
\usepackage{pdflscape}
\usepackage{multirow}
\usepackage{adjustbox,expl3,etoolbox} 
\usepackage[margin = 2.5cm]{geometry}
\usepackage{mathrsfs}
\usepackage{rotating}
\usepackage{tablefootnote}
\usepackage{longtable}
\usepackage{times}
\usepackage{mathpple}

\newtheorem{thm}{Theorem}[section]
\newtheorem{lem}[thm]{Lemma}
\theoremstyle{definition}

\newtheorem{prop}[thm]{Proposition}
\newtheorem{conj}[thm]{Conjecture}

\newtheorem{rmk}[thm]{Remark}
\newtheorem{prob}[thm]{Problem}

\DeclareMathOperator\Aut{Aut}

\DeclareMathOperator\sym{Sym}
\DeclareMathOperator\alt{Alt}
\DeclareMathOperator{\fix}{fix}

\DeclareMathOperator\Der{Der}

\DeclareMathOperator{\cay}{Cay}

\newcommand{\agl}[2]{\operatorname{AGL}_#1(#2)}
\newcommand{\pgl}[2]{\operatorname{PGL}_#1(#2)}
\newcommand{\psl}[2]{\operatorname{PSL}_#1(#2)}
\newcommand{\asl}[2]{\operatorname{ASL}_#1(#2)}

\newcommand{\agammal}[2]{\operatorname{A\Gamma L}_#1(#2)}
\newcommand{\pgammal}[2]{\operatorname{P\Gamma L}_#1(#2)}

\newcommand{\mathieu}[1]{\operatorname{M}_{#1}}

\begin{document}	

\title[]{On the intersection density of the symmetric group acting on uniform subsets of small size}

{\author[A. Behajaina]{Angelot Behajaina\textsuperscript{1}}
	\thanks{\textsuperscript{1}Universit\'e Paris-Saclay, CNRS, Laboratoire de Math\'ematiques d'Orsay, 91405 Orsay, France}
	\email{angelot.behajaina@universite-paris-saclay.fr}}
{\author[R. Maleki]{Roghayeh Maleki\textsuperscript{2}}
	\thanks{\textsuperscript{2} Department of Mathematics and Statistics, University of Regina, Regina, Saskatchewan S4S 0A2, Canada}
	\email{rmaleki@uregina.ca}}

\author[A. S. Razafimahatratra]{Andriaherimanana Sarobidy Razafimahatratra\textsuperscript{2,*}}
\thanks{\textsuperscript{*} Corresponding author}\email{sarobidy@phystech.edu}

\begin{abstract}
	Given a finite transitive group $G\leq \sym(\Omega)$, a subset $\mathcal{F}$ of $G$ is \emph{intersecting} if any two elements of $\mathcal{F}$ agree on some element of $\Omega$. The \emph{intersection density} of $G$, denoted by $\rho(G)$, is the maximum of the rational number $|\mathcal{F}|\left(\frac{|G|}{|\Omega|}\right)^{-1}$ when $\mathcal{F}$ runs through all intersecting sets in $G$. In this paper, we prove that if $G$ is the group $\sym(n)$ or $\alt(n)$ acting on the $k$-subsets of $\{1,2,3\ldots,n\}$, for $k\in \{3,4,5\}$, then $\rho(G)=1$. Our proof relies on the representation theory of the symmetric group and the ratio bound.
\end{abstract}



\date{January 21, 2022}

\maketitle

\section{Introduction}
\subsection{Main results}
Let  $G\leq \sym(\Omega)$ be a finite transitive group. A subset $\mathcal{F	}\subset G$ is \emph{intersecting} if for any $g,h\in \mathcal{F	}$, there exists $\omega\in \Omega$ such that $\omega^g = \omega^h$, or equivalently, $\omega^{hg^{-1}} = \omega$. We are interested in the size and the combinatorial structure of  the largest intersecting sets of $G\leq \sym(\Omega)$. It is easy to see that a coset of a stablizer of a point (i.e., sets of the form $\left\{ g \in G \mid \omega^g = \omega^\prime \right\}$, for some $\omega,\omega^\prime \in \Omega$) is a natural example of an intersecting set in $G$; we will call such intersecting sets \emph{canonical}. Hence, an intersecting set of maximum size in $G$ has size at least $\frac{|G|}{|\Omega|}$.

Let $[n]:=\{1,2,\ldots,n\}$. In 1977, Deza and Frankl found the size of the maximum intersecting set when $G$ is the symmetric group of $\Omega = [n]$, that is, $\sym(n)$. 
\begin{thm}[Deza-Frankl]
	Let $n\geq 3$. If $\mathcal{F} \subset \sym(n)$ is intersecting, then $|\mathcal{F}|\leq (n-1)!$.\label{thm:deza-frankl}
\end{thm}
The largest intersecting sets of $\sym(n)$ were characterized in 2004 by Cameron and Ku \cite{cameron2003intersecting}, and independently, by Larose and Malvenuto \cite{larose2004stable}. Another proof which uses only algebraic arguments was given by Godsil and Meagher in \cite{godsil2009new}.
\begin{thm}
	Let $n\geq 3$. If $\mathcal{F} \subset \sym(n)$ is intersecting of size $(n-1)!$, then $\mathcal{F}$ is a coset of a stabilizer of a point. In particular, there exist $i,j\in \{1,2,\ldots,n\}$ such that
	\begin{align*}
		\mathcal{F} = \left\{ \sigma\in \sym(n) \mid j = i^\sigma \right\}.
	\end{align*} \label{thm:char-sym}
\end{thm}
Therefore, the largest intersecting sets in $\sym(n)$ with its action on $[n]$ are the canonical intersecting sets. For arbitrary transitive groups, however, intersecting sets can have size larger than the order of a stabilizer of a point and the maximum intersecting sets can have a more complex structure than the canonical intersecting sets (see \cite{meagher180triangles} for instance). The smallest example of transitive groups (both in terms of order and degree) whose maximum intersecting sets have size larger than the order of stabilizer of a point is the alternating group $\alt(4)$ acting on the $2$-subsets of $\{1,2,3,4\}$ (see \cite{AMC2554}).

Note that Theorem~\ref{thm:deza-frankl} and Theorem~\ref{thm:char-sym} are generalization of the well-known Erd\H{o}s-Ko-Rado (EKR) Theorem \cite{erdos1961intersection} for the symmetric group.
We say that a finite transitive group $G\leq \sym(\Omega)$ has the \emph{EKR-property} if any intersecting set of $G$ has size at most $\frac{|G|}{|\Omega|}$. Moreover, $G$ has the \emph{strict EKR-property} if any intersecting set of maximum size (i.e. of size $\frac{|G|}{|\Omega|}$) is a coset of a stabilizer of an element of $\Omega$. Examples of groups having the EKR property are the finite $2$-transitive groups \cite{meagher2016erdHos}. However, there are $2$-transitive groups such as $\pgl{3}{q}$ acting on the projective plane \cite{meagher2014erdos} that do not have the strict-EKR property. See \cite{ahmadi2014new,ahmadi2015erdHos,meagher2015erdos,ellis2011intersecting,meagher2019erdHos,meagher2011erdHos,spiga2019erdHos} for other examples concerning EKR and/or strict-EKR properties.

For any transitive group $G\leq \sym(\Omega)$ and an intersecting set $\mathcal{F} \subset G$, the \emph{intersection density} of $\mathcal{F}$ is the number 
\begin{align*}
	\rho(\mathcal{F}) = \frac{|\mathcal{F}|}{|G_\omega|},
\end{align*}
 where $G_\omega$ is the stabilizer of the point $\omega \in \Omega$ in $G$.
The \emph{intersection density} of the group $G$ is the quantity
\begin{align*}
	\rho(G) = \max \left\{ \rho(\mathcal{F}) \mid \mathcal{F} \subset G \mbox{ is intersecting} \right\}.
\end{align*}
The intersection density of a transitive group was introduced in \cite{li2020erd} and was generalized to arbitrary finite permutation groups in \cite{meagher180triangles}. This parameter was introduced to measure how far from having the EKR property a transitive group can be. Indeed, $G\leq \sym(\Omega)$ has the EKR property if and only if $\rho(G) =1$. Moreover, if $\rho(G)= k>1$, then the largest intersecting set of $G$ has size $k$ times the order of a point stabilizer of $G$ (i.e., the size of the canonical intersecting sets). Several papers on the intersection density of transitive groups have recently appeared in the literature \cite{hujdurovic2022intersection,2021arXiv210803943H,hujdurovic2021intersection,li2020erd,AMC2554,razafimahatratra2021intersection}. 

In this paper, we find the intersection density of the groups $\sym(n)$ acting on the $k$-subsets of $[n]$ for $k\in \{4,5\}$, and $\alt(n)$ acting on $k$-subsets of $[n]$ for $k\in \{3,4,5\}$. For $k\in\{2,3\}$, it was proved in \cite{behajaina20203} and \cite{meagher2021erdHos} that the intersection density of $\sym(n)$ acting on the $k$-subsets of $[n]$ is equal to $1$. In \cite{razafimahatratra2021intersection}, it was proved that the group $\alt(n)$ acting on the $2$-subsets of $[n]$ also has intersection density equal to $1$. For a fixed $k\in \mathbb{N}$ and $n$ large enough depending on $k$, Ellis \cite{ellis2012setwise} proved that the intersection density of $\sym(n)$ acting on the $k$-subsets of $[n]$ is equal to $1$.
Our main results are stated as follows.
\begin{thm}
	 If $k \in \{4,5\}$ and $S_k$ is the permutation group $\sym(n)$ acting on the $k$-subsets of $[n]$, then $\rho(S_k) = 1$.\label{thm:main-sym}
\end{thm}

Our proof for Theorem~\ref{thm:main-sym} relies on the Hoffman bound and weighted adjacency matrices. Tough the method used in this paper is essentially similar to the one used in \cite{meagher2021erdHos,razafimahatratra2021intersection}, the eigenvalues coming from low dimensional characters are much harder to control. Using an argument on the multiplicity of the largest eigenvalue of these weighted adjacency matrices, we can also prove the following.

\begin{thm}
	If $k \in \{3,4,5\}$ and $A_k$ is the group $\alt(n)$ acting on the $k$-subsets of $[n]$, then $\rho(A_k) = 1$.\label{thm:main-alt}
\end{thm}

\subsection{Motivation}
Our motivation to prove these two theorems stems from the intersection density of vertex-transitive graphs. We can easily extend the notion of intersection density to vertex-transitive graphs as follows. Given a graph $X = (V,E)$, we say that $X$ is \emph{vertex transitive} if the automorphism group $\Aut(X)$ of $X$ acts transitively on the vertex set $V(X)$ of $X$. If $X = (V,E)$ is a vertex-transitive graph, then we define the \emph{intersection density} of $X$ to be
\begin{align*}
	\rho(X) = \max\left\{ \rho(H) \mid H\leq \Aut(X) \mbox{ is transitive} \right\}.
\end{align*}

Given two positive integers $n$ and $k$ such that $n\geq k$, the \emph{Kneser graph} $K(n,k)$ is the graph with vertex set equal to the collection of all $k$-subsets of $[n]$ and two $k$-subsets $A$ and $B$ of $[n]$ are adjacent if and only if $A\cap B = \varnothing$. The Kneser graph $K(2k,k)$ is a union of edges and $K(n,k)$ is the empty graph whenever $n<2k$. For any $n\geq 4$ and $k\geq 2$ such that $2k<n$, the automorphism group of the Kneser graph $K(n,k)$ is the transitive group $\sym(n)$ acting on the $k$-subsets of $[n]$. The proof of this fact is given in \cite[Chapter~7]{godsil2001algebraic}. Therefore, the study of the intersection density of $K(n,k)$ is equivalent to the study of the intersection density of the $k$\emph{-homogeneous} subgroups of $\sym(n)$ (i.e., subgroups of $\sym(n)$ that are transitive on $k$-subsets of $[n]$). The groups $\sym(n)$ and $\alt(n)$ are always $k$-homogeneous and when $k\geq 6$, they are the only families of groups that are $k$-homogeneous. For $2\leq k \leq 5$, there are various groups that are $k$-homogeneous. See \cite[Theorem~9.4B]{dixon1996permutation} and \cite[Theorem~5.2]{graham1997handbook} for details.
Therefore, the natural starting point in the study of the intersection density of Kneser graphs is the two families, $\alt(n)$ and $\sym(n)$, which are always $k$-homogeneous, for $k\geq 2$.

\subsection{Preliminary reduction}

First, we reduce the problem of finding the intersection density of a transitive group to a problem on the independence number of a graph associated with the group. Recall that if $G$ is a group and $C \subset G\setminus \{1\}$ is \emph{inverse-closed} (i.e., if $x\in C$, then $x^{-1} \in C$), then the \emph{Cayley graph} $\cay(G,C)$ is the graph with vertex set equal to $G$, where two group elements $g$ and $h$ are adjacent if $hg^{-1} \in C$. If $C$ is invariant under conjugation by elements of $G$, then the graph $\cay(G,C)$ is called a \emph{normal Cayley graph}.

Given finite transitive group $G\leq \sym(\Omega)$, the \emph{derangement graph } of $G$, denoted by $\Gamma_G$, is the graph whose vertex set is $G$ and two group elements $g$ and $h$ are adjacent if and only if $hg^{-1}$ is a \emph{derangement} (i.e., a fixed-point-free permutation) of $G$. Equivalently, the derangement graph $\Gamma_G$ is the Cayley graph $\cay(G,\Der(G))$, where $\Der(G)$ is the set of all derangements of $G$. We observe that $g$ and $h$ are intersecting if and only if there exists $\omega \in \Omega$ such that $ \omega^g = \omega^h \Leftrightarrow \omega^{hg^{-1}} = \omega$. We then deduce that $\mathcal{F}$ is an intersecting set in $G$ if and only if it is an \emph{independent set} or a \emph{coclique} in $\Gamma_G$. Therefore, the problem of finding $\rho(G)$ can be reduced to finding $\alpha(\Gamma_G)$ since $\rho(G) = \frac{\alpha(\Gamma_G)}{|G_\omega|}$, where $\omega \in \Omega$. 

Given two intersecting permutations $\sigma,\pi \in \sym(n)$ in its action on the $k$-subsets of $[n]$, there exists a $k$-subset $S$ of $[n]$ such that $S^\pi = S^\sigma$. That is, $\sigma$ and $\pi$ agree on a $k$-subset of $[n]$. We say that two permutations $\sigma$ and $\pi$ of $\sym(n)$ are $k$\emph{-setwise intersecting} if they agree on a $k$-subset of $[n]$ and more generally, we say that $\mathcal{F} \subset \sym(n)$ is $k$\emph{-setwise intersecting} if any two permutations in $\mathcal{F}$ are $k$-setwise intersecting. Consequently, studying the intersection density of the group $\sym(n)$ acting on the $k$-subsets of $[n]$ is equivalent to studying the $k$-setwise intersecting sets of $\sym(n)$.

Recall that a \emph{partition} $\lambda = [\lambda_1,\lambda_2,\ldots,\lambda_k]$ of the integer $n$ is a non-increasing sequence of positive integers summing to $n$. That is, $\lambda_1\geq\lambda_2\geq\ldots\geq \lambda_k \geq 1$ and $\lambda_1+\lambda_2+\ldots+\lambda_k = n$. A partition  $\lambda$ of $n$ will be denoted by $\lambda \vdash n$.

 Given a $k$-subset $S$ of $[n]$, the stabilizer of $S$ in $\sym(n)$ is conjugate to $\sym(k) \times \sym(n-k)$. Therefore, a derangement in $\sym(n)$ acting on the $k$-subsets of $[n]$ is a permutation that does not have a subpartition of $k$ in its cycle type. That is, if $\sigma\in \sym(n)$ is a permutation with cycle type $\lambda = (\lambda_1,\lambda_2,\ldots,\lambda_t) \vdash n$, then $\sigma$ is a derangement for the action of $\sym(n)$ on the $k$-subsets of $[n]$ if and only if there is no $\mu = (\mu_1,\mu_2,\ldots,\mu_\ell) \vdash k$ such that $\mu_i \in \left\{\lambda_1,\lambda_2,\ldots,\lambda_t \right\}$, for all $i\in \{1,2,\ldots,\ell\}$. A derangement for the action of $\sym(n)$ on the $k$-subsets of $[n]$ is called a $k$\emph{-derangement} and we denote the set of all $k$-derangements of $\sym(n)$ by $D_{n,k}$.

Using the information from the two previous paragraphs, one can deduce that the derangement graph of $\sym(n)$ acting on the $k$-subsets of $[n]$ is the Cayley graph $\Gamma_{n,k} := \cay(\sym(n),D_{n,k})$. Using the above-mentioned reduction, an \emph{independent set} or a \emph{coclique} of $\Gamma_{n,k}$ is exactly a $k$-setwise intersecting set of $\sym(n)$, and vice versa. Therefore, algebraic graph theory results such as the Hoffman bound can be applied.

\subsection{Organization}
This paper is organized as follows. First, in Section~\ref{sect:background} and Section~\ref{sect:rep-theory} we give the background materials that are needed in our proofs of the main results. In Section~\ref{sect:idea}, we give the main ideas of our proofs. The sections after this are devoted to the proof of Theorem~\ref{thm:main-sym} and Theorem~\ref{thm:main-alt}. In Section~\ref{sect:application}, we prove results on the intersection density of the Kneser graphs $K(n,4)$ and $K(n,5)$.

\section{The Ratio bound}\label{sect:background}
We let $G\leq \sym(\Omega)$ be a finite transitive group. 
Since the problem of finding the largest sets of intersecting permutations of the group $G$ can be reduced to the independence number of the derangement graph, we present a standard spectral upper bound on the latter. 

Recall that $X = (V,E)$ is $k$-regular if the degree of each vertex of $X$ is equal to $k$. It is easy to see that the all-ones vector $\mathbf{1}$ (with the appropriate dimension) is an eigenvector with eigenvalue $k$ of the adjacency matrix of $X$. In fact, $k$ is the largest eigenvalue of $X$. Since the derangement graph $\Gamma_{n,k} = \cay(\sym(n),D_{n,k})$ defined in the previous section is a Cayley graph, it is regular with degree $|D_{n,k}|$.

In the following lemma, we present the famous ratio bound. See \cite{haemers2021hoffman} for the history of this bound.

\begin{lem}[ratio bound or Hoffman bound]
	Let $X = (V,E)$ be a $k$-regular graph with minimum eigenvalue $\tau$. The independence number of $X$ is such that
	\begin{align*}
		\alpha(X) \leq \frac{|V(X)|}{1-\frac{k}{\tau}}.
	\end{align*} \label{lem:ratio-bound}
\end{lem}

As noted in \cite[Chapter~2]{godsil2016erdos}, the proof of Lemma~\ref{lem:ratio-bound} does not use the fact that the adjacency matrix of the regular graph $X$ is a $\{0,1\}$-matrix. Therefore, Lemma~\ref{lem:ratio-bound} can be generalized further (see \cite[Theorem~2.4.2]{godsil2016erdos} for the generalization). First, we will define weighted adjacency matrices and then we will state this generalization for derangement graphs.

Let $\mathcal{C}$ be the set of all conjugacy classes of $G$. For any $C\in \mathcal{C}$, let $A_C$ be the $|G|\times |G|$ matrix with entries from $\{0,1\}$ and indexed by the group elements in its rows and columns such that the entry $(g,h)$ has the property
\begin{align*}
	A_C(g,h) =1 \Leftrightarrow hg^{-1} \in C.
\end{align*}
Let $\mathcal{A}(G) = \left\{ A_C \mid C\in \mathcal{C} \right\}$. The set of matrices $\mathcal{A}(G)$ forms an \emph{association scheme} called the \emph{conjugacy class scheme} of $G$. For more information on this scheme, see \cite[Section~3.3]{godsil2016erdos}. When $G = \sym(n)$, the conjugacy class of permutations with cycle type $\lambda \vdash n$ is denoted by $C_\lambda$ and the corresponding matrix in the association scheme $\mathcal{A}(\sym(n))$ is $A_\lambda$. To avoid confusion, we will use parentheses for the cycle types of elements of the symmetric group.

It is not hard to see that the adjacency matrix of the derangement graph $\Gamma_G$ is a sum of matrices of $\mathcal{A}(G)$. Precisely, if $D_1,D_2,\ldots,D_k \in \mathcal{C}$ are the conjugacy classes of derangements of $G$, then the adjacency matrix of $\Gamma_G$ is 
\begin{align*}
	\sum_{i=1}^k A_{D_i}.
\end{align*}

Since the edges of $\Gamma_G$ are governed by the derangements, it is sometimes possible to improve Lemma~\ref{lem:ratio-bound} for derangement graphs by assigning weights to edges. For derangement graphs, we can find an upper bound on the independence number from the conjugacy class scheme of the underlying group. We say that a $|G|\times |G|$ symmetric matrix $A$ (indexed by the group elements in its rows and columns) is a \emph{weighted adjacency matrix} of a subgraph $\Gamma$ of $\Gamma_G$ if  $A$ has a constant row sum and  
\begin{align*}
	g\not\sim_{\Gamma} h \Rightarrow A(g,h) = 0.
\end{align*}
Note that $g$ and $h$ need not be non-adjacent in $\Gamma$ when $A(g,h) = 0$. It is not hard to see that a linear combination of the classes corresponding to derangements in $\mathcal{A}(G)$ is a weighted adjacency matrix of some spanning subgraph of $\Gamma_G$. If  $D_1,D_2,\ldots,D_k \in \mathcal{C}$ are the conjugacy classes of derangements of $G=\sym(n)$ and $A = \sum_{i=1}^k \omega_i A_{D_i}$, then $A$ is a weighted adjacency matrix of the graph 
\begin{align*}
	\cay(\sym(n),\cup_{\{i \mid \omega_i \neq 0\}} D_i).
\end{align*}

\begin{rmk}
	Note that when $G = \sym(n)$, then the conjugacy classes are inverse-closed. Consequently, the matrices in $\mathcal{A}(G)$ are all symmetric. Therefore, the eigenvalues of any weighted adjacency matrix of a normal Cayley graph over $G$ are real numbers. This is not true in general for an arbitrary group $G$. To obtain real eigenvalues, the matrices that correspond to the conjugacy classes that contains $x$ and the one containing $x^{-1}$ need to be weighted equally.
\end{rmk}

For the remainder of this section, we shall assume that $G\leq \sym(\Omega)$ has inverse-closed conjugacy classes. The following lemma is the generalization of the ratio bound for derangement graphs. The proof follows from \cite[Theorem~2.4.2]{godsil2016erdos}.
\begin{lem}
	Let $D_1,D_2,\ldots,D_k$ be all the conjugacy classes of derangements of $G\leq \sym(\Omega)$ and let $A$ be in the span of $\left\{A_{D_1},A_{D_2},\ldots,A_{D_k}\right\}$. Assume that the row sum of $A$ is $k$ and its smallest eigenvalue is $\tau$. If $A$ is a weighed adjacency matrix of a subgraph $\Gamma$ of $\Gamma_G$, then 
	\begin{align*}
		\alpha(\Gamma) \leq \frac{|V(X)|}{1 - \frac{k}{\tau}}.
	\end{align*}\label{lem:wRatioBound}
\end{lem}

Now, we obtain an upper bound on the independence number of $\Gamma_G$ using the simple observation that $\Gamma$  is a spanning subgraph of $\Gamma_G$. Therefore, $\alpha(\Gamma_G) \leq \alpha(\Gamma)$ and  Lemma~\ref{lem:wRatioBound} yields an upper bound on the independence number of $\Gamma_G$.  

Since Lemma~\ref{lem:wRatioBound} depends on eigenvalues, we present a result on the eigenvalues of the weighted adjacency matrices of $\Gamma_G$. Its proof is a result of the existence of a basis of idempotents of the matrix algebra generated by $\mathcal{A}(G)$, i.e., the Bose-Mesner algebra (see \cite[Section~3.4]{godsil2016erdos} for details).

\begin{lem}
	Let $D_1,D_2,\ldots,D_k$ be all the conjugacy classes of derangements of $G\leq \sym(\Omega)$ and let $A = \displaystyle\sum_{i=1}^k \omega_i A_{D_i}$, where $\omega_i \in \mathbb{R}$ for all $i\in [k]$. Then, the spectrum of $A$ is determined by the irreducible characters of $G$. In particular, any irreducible character $\chi$ of $G$ corresponds to the eigenvalue
	\begin{align*}
		\frac{1}{\chi(id)} \sum_{i=1}^k \omega_i |D_i|\chi(g_i),
	\end{align*} 
	where $g_1,g_2,\ldots,g_k$ are elements of the conjugacy classes of derangements $D_1,D_2,\ldots,D_k$, respectively.\label{lem:eigenvalues-scheme}
\end{lem}

\section{Representation theory}\label{sect:rep-theory}
In this section, we present the main tools that we need in our proof of the main results. We assume that the reader is familiar with the representation theory of the symmetric group, as we will only give the main results. We  refer the readers to the textbook on the representation theory of the symmetric group by Sagan \cite{sagan2001symmetric}.

\subsection{Low-dimensional irreducible characters}
It is well-known that the irreducible representations of $\sym(n)$ are uniquely determined by the partitions of $n$. If $\lambda \vdash n$, then the corresponding irreducible $\mathbb{C}\sym(n)$-module is the \emph{Specht} module $S^\lambda$. If $\mathfrak{X}^\lambda$ is the matrix representation equivalent to $S^\lambda$, then the corresponding irreducible character is denoted by $\chi^\lambda$. For any $\lambda \vdash n$ and permutation $\sigma \in \sym(n)$ in a conjugacy class $C$ with cycle-type $\rho$, we define simultaneously $\chi_\rho^\lambda$ and $\chi^\lambda(C)$ to be the value $\chi^\lambda(\sigma)$. The \emph{dimension} of the irreducible character $\chi^\lambda$ is the natural number $f^\lambda:= \chi^\lambda(id)$, where $id$ is the identity map in $\sym(n)$.

Given $\lambda\vdash n$, the partition $\lambda^\prime$ is the partition of $n$ obtained from transposing the Young diagram corresponding to $\lambda$. It is well known that if $S^\lambda$ is the Specht module corresponding to $\lambda\vdash n$, then $S^{\lambda^\prime} = S^{[1^n]} \otimes S^{\lambda}$. Therefore, the identity $f^\lambda = f^{\lambda^\prime}$ follows.

The next three lemmas give the irreducible characters of dimension at most $\binom{n}{4}$, $\binom{n}{5}$, and $2\binom{n}{6}$. 

\begin{lem}
	Let $n\geq 15$ and $\lambda\vdash n$. If $f^\lambda < \binom{n}{4}$, then $\lambda \in S_n$ or $\lambda^\prime \in S_n$, where 
	\begin{align*}
		S_n = \left\{ 
		\begin{matrix}
			[n],\\
			[n-1,1],\\
			[n-2,2],[n-2,1^2],\\
			[n-3,3],[n-3,2,1], [n-3,1^3],\\
			[n-4,4],[n-4,1^4]
		\end{matrix}
		  \right\}.
	\end{align*}\label{lem:characters-dim-4}
\end{lem}

\begin{lem}
	Let $n\geq 19$ and $\lambda \vdash n$. If $f^\lambda <\binom{n}{4}$, then $\lambda$ is one of the partitions given in Lemma~\ref{lem:characters-dim-4}. If $\binom{n}{4}<f^\lambda <\binom{n}{5}$, then $\lambda\in T_n$ or $\lambda^\prime \in T_n$, where
	\begin{align*}
		T_n =\left\{
		\begin{matrix}
			[n-4,3,1], [n-4,2^2], [n-4,2,1^2],\\
			[n-5,5],[n-5,1^5]
		\end{matrix}
		\right\}.
	\end{align*}\label{lem:characters-dim-5}
\end{lem}

\begin{lem}
	Let $n\geq 27$ and $\lambda \vdash n$. If $f^\lambda <\binom{n}{5}$, then $\lambda$ is one of the partitions given in Lemma~\ref{lem:characters-dim-4} and Lemma~\ref{lem:characters-dim-5}. If $\binom{n}{5}<f^\lambda <2\binom{n}{6}$, then $\lambda\in U_n$ or $\lambda^\prime \in U_n$, where
	\begin{align*}
	U_n =\left\{
	\begin{matrix}
	[n-5,4,1], [n-5,3,2],[n-5,3,1^2], [n-5,2^2,1], [n-5,2,1^3],\\
	[n-6,6],[n-6,1^6]
	\end{matrix}
	\right\}.
	\end{align*}\label{lem:characters-dim-6}
\end{lem}

We will only give a proof of the last lemma, since the proof of the first two are similar. A similar approach can be used to prove Lemma~\ref{lem:characters-dim-4} and Lemma~\ref{lem:characters-dim-5} using Table~\ref{tab:k4} and Table~\ref{tab:k5}, respectively.

\begin{proof}
	From Table \ref{tab:caseee}, one can easily observe that the degree of characters in   $S_n \cup T_n \cup U_n$ are strictly less than $2\binom{n}{6}$ when $n \geq 27$. 
	
	For the converse, we use a similar argument to the proof of \cite[Lemma~12.7.3]{godsil2016erdos} by using induction on $n$. Computation via \verb|Sagemath| \cite{sagemath} shows that the result in Lemma~\ref{lem:characters-dim-6} holds for $n=27,28$. Now, suppose that the lemma is true for $n$ and $n-1$, and let us prove the induction step. Let $\phi = \chi^\mu$, where $\mu \vdash n+1$, be an irreducible character of $\sym(n+1)$ of dimension less than $2\binom{n+1}{6}$. Denote by $\phi\downharpoonleft _{[n]}$ (resp., $\phi\downharpoonleft _{[n-1]}$) the restriction of $\phi$ to $\mathrm{Sym}(n)$ (resp., $\mathrm{Sym}(n-1)$).
	
	First, assume that $\phi\downharpoonleft _{[n]}$ has a constituent $\chi^\lambda$ such that $\lambda \in S_n\cup T_n \cup U_n$. From the branching rule (see \cite[ Theorem~2.8.3]{sagan2001symmetric}), $\mu$ is one of the partitions given in the second column of Table \ref{tab:poss}. As the dimension of any irreducible character corresponding to a partition in Table~\ref{tab:k6} is larger than $2\binom{n+1}{6}$ when $n \geq 27$, we deduce that $\mu$ is in $S_{n+1}\cup T_{n+1} \cup U_{n+1}$. Similarly, if $\phi\downharpoonleft _{[n]}$ has a constituent $\chi^\lambda$ such that $\lambda^\prime \in S_n\cup T_n \cup U_n$, then $\mu^\prime \in S_{n+1}\cup T_{n+1} \cup U_{n+1}$.
	
	Define $S_n^\prime = \{ \lambda^\prime \mid  \lambda \in S_n  \}$, $T^\prime_n = \left\{ \lambda^\prime \mid \lambda \in T_n \right\}$ and $U^\prime_n = \left\{ \lambda^\prime \mid \lambda \in U_n \right\}$. Assume that $\phi\downharpoonleft _{[n]}$ has no constituents listed in $S_n \cup T_n \cup U_n\cup S_n' \cup T_n' \cup \mathcal{U}_n'$ and suppose that $\phi\downharpoonleft _{[n]}$ has at least two constituents. 
	Then, the dimension of $\phi\downharpoonleft _{[n]}$ is at least $4\binom{n}{6}> 2\binom{n+1}{6}$, whenever $n \geq 27$. Hence, the dimension of $\phi$ is at least $4\binom{n}{6}>2\binom{n+1}{6}$. This is a contradiction.
	
	Finally, we may assume that $\phi\downharpoonleft _{[n]}$ is irreducible and is not in $S_n \cup T_n \cup U_n \cup S_n' \cup T_n' \cup U_n'$. From the branching rule, the Young diagram corresponding to $\phi$ (i.e. $\mu$) must be rectangular, that is, $\mu=[a^b]$ for some positive integers $a,b \geq 1$. Again, using the branching rule, we can deduce that the constituents of $\phi\downharpoonleft _{[n-1]}$ correspond to the two partitions $\lambda_0=[a^{b-1},a-2]$ and $\lambda_1=[a^{b-2},a-1,a-1]$. None of $\lambda_0$, $\lambda_1$, $\lambda_0'$ or $\lambda_1'$ is in $S_{n-1} \cup T_{n-1} \cup U_{n-1}$, when $n \geq 27$.
	Therefore, the degree of $\phi$ is at least $4\binom{n-1}{6} > 2\binom{n+1}{6}$ when $n \geq 27$. Again, this is a contradiction.
	This completes the proof.
\end{proof}

We make the following conjecture about the dimensions of the irreducible characters of $\sym(n)$ of degree at most $\binom{n}{k}$.
\begin{conj}
	For $k\geq 2$ fixed and $n$ large enough, if $\lambda = [n-a,\lambda_2,\ldots,\lambda_t] \vdash n$ such that $t<n-a$ and the irreducible character of $\sym(n)$ afforded by $\lambda$ has dimension $f^\lambda < \binom{n}{k}$, then $a \leq k-1$ or $\lambda \in \{ [n-k,k],[n-k,1^k] \}$.
\end{conj}

\subsection{Permutation character}
Let $G \leq \sym(\Omega)$ be a finite permutation group. Recall that the \emph{permutation character} of $G$ is the character $\fix$, which is the map from $G$ to $\mathbb{C}$ such that
\begin{align*}
	\fix(g) &= | \left\{ \omega \in \Omega \mid \omega^g = \omega \right\}|,
\end{align*}
for $g\in G$. That is, $\fix(g)$ is the number of fixed points of $g \in G$. We present the following straightforward observation about the permutation character of $G$.
\begin{prop}
	Let $G\leq \sym(\Omega)$ be a finite permutation group. Assume that the irreducible constituents of $\fix$ are $\chi_0,\chi_1,\ldots,\chi_\ell$, where $\chi_0$ is the trivial character of $G$ \footnote{The trivial character of $G$ is always a constituent of $\fix$ since $\langle \fix,\chi_0\rangle = \frac{1}{|G|} \sum_{g\in G}\fix(g)$ is the number of orbits of $G$.}, with multiplicities $m_0,m_1,\ldots,m_\ell$, respectively. Then, we have
	\begin{align*}
		\sum_{i=1}^\ell m_i\chi_i(1) = |\Omega|-m_0.
	\end{align*}
\end{prop}
\begin{proof}
	Since $\fix = \sum_{i=0}^\ell m_i \chi_i$, we have 
	\begin{align*}
		|\Omega|=\fix(1) &= \sum_{i=0}^\ell m_i \chi_i(1) = m_0 + \sum_{i=1}^\ell m_i \chi_i(1).
	\end{align*}
\end{proof}

In the next lemma, we give the irreducible constituents of the permutation character of the group $\sym(n)$ acting on the $k$-subsets of $[n]$. A proof of this lemma is given in \cite[Proposition~1.6~(a)]{godsil2010multiplicity}.

\begin{lem}
	The permutation character of $\sym(n)$ acting on the $k$-subsets of $[n]$ is given by
	\begin{align*}
		\sum_{i=0}^k\chi^{[n-i,i]}.
	\end{align*}
\end{lem}

\section{Main idea of the proof}\label{sect:idea}

 We recall the following well-known propositions (see \cite{godsil2001algebraic} for details).

\begin{prop}
	If $X=(V,E)$ is a $d$-regular graph, then the number of components of $X$ is equal to the multiplicity of the largest eigenvalue $d$.\label{prop:components}
\end{prop}
\begin{prop}
	Let $\cay(G,C)$ be a Cayley graph. Then, $\cay(G,C)$ is connected if and only if $\langle C\rangle = G$. \label{prop:connected-cay}
\end{prop}

In the next proposition, we prove that the eigenvalues  corresponding to the partitions $\lambda\vdash n$ and its transpose $\lambda^\prime \vdash n$ of a normal Cayley graph of the symmetric group whose connection set consists only of even permutations, are equal. 

\begin{prop}
	Let $n\geq 4$ and let  $\cay(\sym(n),C)$ be a normal Cayley graph. Assume that $A = \sum_{\mu \vdash n} \omega_\mu A_\mu$ is a weighted adjacency matrix of $\cay(\sym(n),C)$\footnote{That is, if $\lambda \vdash n$ is such that $\omega_\lambda\neq 0$, then the conjugacy class with cycle type $\lambda$ is a subset of $C$}. If the values of the characters corresponding to $[n] $ and $ [1^n]$ are equal on the set 
	\begin{align*}
	T = \bigcup_{\{ \lambda\vdash n \mid \omega_\lambda \neq 0 \}}\left\{ \sigma \in \sym(n) \mid \sigma \mbox{ has cycle type }\lambda \right\},
	\end{align*}
	then for all $\lambda \vdash n$, the eigenvalues of $A$ afforded by $\lambda$ and $\lambda^\prime$ are equal.
\end{prop}
\begin{proof}
	Since $\chi^{[n]}$ and $\chi^{[1^n]}$ agree on $T$, every element of $T$ is an even permutation. The rest of the proof is an immediate consequence of $S^{\lambda^\prime} = S^{[1^n]} \otimes S^\lambda$.
\end{proof}

To prove Theorem~\ref{thm:main-sym} and Theorem~\ref{thm:main-alt}, we use weighted adjacency matrices from the conjugacy class scheme of $\sym(n)$. In particular, we will assign weights on the conjugacy classes of $k$-derangements in $\sym(n)$. Such a weighted adjacency matrix will correspond to a spanning subgraph of the graph $\Gamma_{n,k}$ which is a Cayley graph $\cay(\sym(n), C)$, where $C$ is a union of conjugacy classes of $k$-derangements. Since $\langle C\rangle$ is a normal subgroup of $\sym(n)$ and $n\geq 5$, we conclude that $\alt(n) \leq \langle C\rangle \leq \sym(n)$. In fact, it is easy to see that $\langle C \rangle  = \alt(n)$ if and only if the conjugacy classes of $k$-derangements in $C$ consist of even permutations.

Hence, we deduce that $\cay(\sym(n),C)$ has at most two components. By Proposition~\ref{prop:components} and Proposition~\ref{prop:connected-cay}, the multiplicity of the largest eigenvalue is equal to $2$ whenever $\langle C\rangle = \alt(n)$ and equal to $1$ if $\langle C\rangle = \sym(n)$. In the case that $\langle C\rangle = \alt(n)$, by vertex-transitivity of a Cayley graph, one can see that $\cay(\sym(n),C)$ is equal to a disjoint union of two copies of $\cay(\alt(n),C)$. Therefore,
\begin{align}
	\alpha \left(\cay(\sym(n), C)\right) = 2 \alpha \left(\cay(\alt(n),C)\right).\label{eq:identity-cocliques}
\end{align}
Moreover, the spectrums of $\cay(\sym(n),C)$ and $\cay(\alt(n),C)$ are equal up to multiplicity. In particular, $\theta$ is an eigenvalue with multiplicity $\ell$ in $\cay(\alt(n),C)$ if and only if $\theta$ is an eigenvalue with multiplicity $2\ell$ in $\cay(\sym(n),C)$. 

Consequently, if $d$ and $\tau$ are the largest and smallest eigenvalues of a weighted adjacency matrix corresponding to $\cay(\sym(n),C)$, respectively, then the weighted ratio bound (see Lemma~\ref{lem:wRatioBound}) on $\cay(\sym(n),C)$ yields
\begin{align}
	\alpha(\cay(\sym(n),C)) \leq \frac{n!}{1 - \frac{d}{\tau}}.\label{eq:bound-sym}
\end{align}
If the multiplicity of $d$ is equal to $2$ (or equivalently, $C$ consists of even permutations), then by \eqref{eq:identity-cocliques}, we have
\begin{align}
	\alpha(\cay(\alt(n),C)) \leq\frac{n!}{2\left(1 - \frac{d}{\tau}\right)} =  \frac{\frac{n!}{2}}{1 - \frac{d}{\tau}}.\label{eq:bound-alt}
\end{align}

Therefore, if the largest eigenvalue of a weighted adjacency matrix of $\cay(\sym(n),C)$ has multiplicity equal to $2$, then we obtain a bound on the size of the maximum coclique of $\cay(\alt(n),C)$. This is the main idea of our proof.

Our proof is almost similar to the proof of the main result in \cite{behajaina20203} and \cite{meagher2021erdHos} in the sense that we use eigenvalue methods and the representation theory of the symmetric group. However, the technique used in our proof is more refined for we obtain Theorem~\ref{thm:main-alt}. Let $n\geq 6$ and $2k<n$. We will find a weighted adjacency matrix $A$ which is a linear combination of matrices in the conjugacy class scheme of $\sym(n)$ consisting of $k$-derangements and having the following properties.
\begin{enumerate}[(i)]
	\item The largest eigenvalue of $A$ is equal to $\binom{n}{k}-1$ and is afforded by the irreducible characters corresponding to $[n]$ and $[1^n]$,\label{first}
	\item the smallest eigenvalue of $A$ which is equal to $-1$ is afforded by the characters corresponding to the partitions $[n-i,i]$ for $i\in \{1,2,\ldots,k\}$,\label{second}
	\item All other eigenvalues of $A$ are at least $-1$.\label{third}
\end{enumerate}
The weighted adjacency matrix $A$ that we will choose will also have the additional property that most of the eigenvalues corresponding to irreducible characters of dimension less than $\binom{n}{k-1}$ are equal to $-1$.

It is clear that any graph $X = \cay(\sym(n),C)$ corresponding to the weighted adjacency matrix $A$ (i.e., $X$ is a spanning subgraph of $\Gamma_{n,k}$) satisfying \eqref{first} must be disconnected into two components. Let $Y = \cay(\alt(n),C)$ be the component of $X$ containing the identity of $\sym(n)$. Since \eqref{second} and \eqref{third} are also satisfied, we deduce from \eqref{eq:bound-sym} and \eqref{eq:bound-alt} that 
\begin{align*}
	\alpha(X) \leq k!(n-k)! \mbox{ and } \alpha(Y) \leq \frac{1}{2}k!(n-k)!.
\end{align*}

For each $k\in \{4,5\}$, we will utilize the above-mentioned method to prove Theorem~\ref{thm:main-sym}. We will then obtain Theorem~\ref{thm:main-alt}, for $k\in \{4,5\}$, as explained earlier in \eqref{eq:bound-alt}. The case $k = 3$ for Theorem~\ref{thm:main-alt} will be considered separately.

\section{$3$-setwise for $\alt(n)$}

In this section, we prove the case $k=3$ of Theorem~\ref{thm:main-alt}. We recall the following theorem.
\begin{thm}[\cite{behajaina20203}]
	For any $n\geq 4$, if $\mathcal{F} \subset \sym(n)$ acting on the $3$-subsets of $[n]$, then $|\mathcal{F}|\leq 6(n-3)!$.\label{thm:3-setwise}
\end{thm}

The proof of this theorem involves a weighted adjacency matrix. 
In this section, we extend Theorem~\ref{thm:3-setwise} for the alternating group. We prove the following.
\begin{thm}
	For any $n\geq 7$, if $\mathcal{F} \subset \alt(n)$ acting on the $3$-subsets of $[n]$, then $|\mathcal{F}| \leq 3(n-3)!$.\label{thm:3-setwise-alt}
\end{thm}

To prove this theorem, we use the idea given in Section~\ref{sect:idea}. 
\subsection{Small cases}
 First, we consider the small cases with \verb|Sagemath|. For $7 \leq n \leq 25$, we verified that there exists a weighted adjacency matrix of $\Gamma_{n,3}$ that satisfies \eqref{first}, \eqref{second} and \eqref{third}. The \verb|Sagemath| code for this is available in \cite{link}.

We give two different weighted adjacency matrices in the next two subsections, depending on the parity of $n$.
\subsection{A weighted adjacency matrix for $n$ odd}
First, we recall the weighted adjacency matrix used in \cite{behajaina20203} and then we use this matrix to prove Theorem~\ref{thm:3-setwise-alt}

The weights that were used, for $n$ odd, in the proof of Theorem~\ref{thm:3-setwise} are given next.
Consider the weighted adjacency matrix
\begin{align}
A &= x_1 A_{(n)} +x_2 A_{(n-2,1^2)} +x_3 A_{(n-2,2)} + x_4 A_{(n-5,4,1)} + x_5 A_{(n-1,1)}.\label{eq:odd-adj}
\end{align}

Let $C_1:= C_{(n)},\ C_2 := C_{(n-2,1^2)} ,\ C_3 := C_{(n-2,2)},\ C_4 := C_{(n-5,4,1)}$, and $C_5 := C_{(n-1,1)}$. Let $\alpha= \binom{n}{3}-1$ and let $\beta,\gamma$, and $\delta$ be, respectively, the degree of the Specht modules corresponding to $\ [n-1,1],\ [n-2,2]$, and $[n-3,3]$. That is, $\beta = (n-1)$, $\gamma = \binom{n}{2}-n$, and $\delta = \binom{n}{3} - \binom{n}{2}$.

Using Lemma~\ref{lem:eigenvalues-scheme}, it is easy to see that the eigenvalues corresponding to $A$ are of the form
\begin{align}
\xi_{\lambda}(s,t) &= \frac{1}{f^\lambda} \sum_{i=1}^5 x_i |C_i| \chi^\lambda(C_i)\label{eq:general_eigenvalues_odd},
\end{align}
where $\lambda$ runs through all partitions of $n$. For any $i\in \{1,2,\ldots,5\}$, define $\omega_i:= |C_i|x_i$.

For $n\geq 27$ odd, the weights  used in \cite{behajaina20203} are

\begin{align}
\begin{cases}
\omega_1(s,t) &= -s- t + (\beta +\gamma)\\
\omega_2(s,t) &= -\frac{1}{2}s-\frac{1}{2} t + \frac{1}{2} (\alpha -\beta)\\
\omega_3(s,t) &= \frac{1}{2}s + \frac{1}{2}t + \frac{1}{2} (\alpha -\beta) -\gamma\\
\omega_4(s,t) &= s\\
\omega_5(s,t) &= t
\end{cases}\label{eq:solution_odd}
\end{align}
where $(t,s) \in \mathbb{R}^2$ belongs to the polytope
\begin{align}
\begin{split} 
\left\{ 
\begin{aligned}
& 3x + y  < \beta +\gamma, \\
&-\frac{n(n-2)(n-4)}{3} < y-x \leq \beta + \gamma - \binom{n-1}{3}, \\
&\beta + \gamma - \binom{n-1}{3} \leq x + y < \beta + \gamma .
\end{aligned}
\right.
\end{split} 	\label{eq:polytope-odd}
\end{align}

When $n$ is odd, the conjugacy classes with cycle type $(n-2,1^2), \ (n),\ (n-5,4,1)$ consist of  even permutations, whereas the conjugacy classes with cycle type $(n-2,2)$ and $(n-1,1)$ consist of odd permutations. It is clear that for the same weighted adjacency matrix given in \eqref{eq:odd-adj} with the weights in \eqref{eq:solution_odd} to satisfy \eqref{first}, \eqref{second} and \eqref{third}, the following table needs to be satisfied (the \checkmark means that the weight need not be equal to $0$).

\begin{table}[H]
	\begin{tabular}{|c|c|c|c|c|c|}
		\hline
		Parity of $n$ & $(n)$ & $(n-2,1^2)$& $(n-2,2)$& $(n-5,4,1)$& $(n-1,1)$ \\ \hline 
		Odd & \checkmark & \checkmark & $0$ & \checkmark & $0$\\ \hline
	\end{tabular}
	\caption{Desired weights for $n$ odd}\label{tab:desired-weight-odd}
\end{table}

Now, we use the previous table to prove Theorem~\ref{thm:3-setwise-alt} for $n$ odd. We let $t_0 = 0$ and $s_0 = \gamma -\delta$. We have
\begin{align}
\begin{cases}
\omega_1(s_0,t_0) &= \beta +\delta\\
\omega_2(s_0,t_0) &= \delta\\
\omega_3(s_0,t_0) &= 0\\
\omega_4(s_0,t_0) &= \gamma - \delta\\
\omega_5(s_0,t_0) &= 0
\end{cases}\label{eq:desired-weight-odd}
\end{align}
It is clear that \eqref{eq:desired-weight-odd} satisfies Table~\ref{tab:desired-weight-odd}. Now, we verify that $(t_0,s_0)$ belongs to the polytope in \eqref{eq:polytope-odd}. It is clear that the first halfspace in \eqref{eq:polytope-odd}  contains $(t_0,s_0)$. By noting that $s_0 = \gamma -\delta = \beta +\gamma - \binom{n-1}{3}$, it is also clear that $(t_0,s_0)$ belongs to the last two halfspaces of \eqref{eq:polytope-odd}.

Therefore, the eigenvalues of the weighted adjacency matrix $A$ given in \eqref{eq:odd-adj} for $(t,s) = (t_0,s_0)$ are in the interval $\left[ -1,\binom{n}{3}-1 \right]$. Moreover, the eigenvalue corresponding to the irreducible character $\chi^{[n]} $ is $\binom{n}{3}-1$ and the eigenvalues corresponding to $\chi^{[n-1,1]},\ \chi^{[n-2,2]}$ and $\chi^{[n-3,3]}$ are all equal to $-1$. By \cite[Equation~7]{behajaina20203}, the eigenvalue corresponding to the character $\chi^{[1^n]}$ is
\begin{align*}
(-1)^{n-1} (-s_0-3t_0 + \beta +2\gamma ) = \delta + \beta +\gamma = \binom{n}{3}-1.
\end{align*}
Therefore, the multiplicity of the largest eigenvalue of $A$ is equal to $2$.

\subsection{A weighted adjacency matrix for $n$ even}
In this section we prove Theorem~\ref{thm:3-setwise-alt} for $n\geq 20$ even. Similar to the proof for $n$ odd, we use the weighted adjacency matrix in \cite{behajaina20203} and we will give a specific weighting corresponding to a disconnected (spanning) subgraph of $\Gamma_{n,3}$. 

The weighted adjacency matrix used in \cite{behajaina20203} for the action of the symmetric group $\sym(n)$ on $3$-subsets of $[n]$ is 
\begin{align}
A &= x_1 A_{(n-5,5)} + x_2 A_{(n-6,2^3)} + x_3 A_{(n-6,4,1^2)} + x_4 A_{(n-6,4,2)} + x_5 A_{(n-6,5,1)}.\label{eq:weighted_adjacency_matrix_even}
\end{align}

We define $C_1 := C_{(n-5,5)},\ C_2 := C_{(n-6,2^3)},\ C_3:= C_{(n-6,4,1^2)},\ C_4 := C_{(n-6,4,2)}$, and $C_5 := C_{(n-6,5,1)}$. 
By Lemma~\ref{lem:eigenvalues-scheme}, the eigenvalues corresponding to $A$ are of the form
\begin{align}
\xi_{\lambda}(s,t) &= \frac{1}{f^\lambda} \sum_{i=1}^5 x_i |C_i| \chi^\lambda(C_i),\label{eq:general_eigenvalues_even}
\end{align}
where $\lambda \vdash n$.

For $i\in \{1,2,3,4,5\}$, we define $\omega_i = |C_i| x_i$.  Let $\alpha = \binom{n}{3}-1$ and let $\beta,\gamma, \mbox{ and }\delta$ be respectively the degree of the irreducible characters $\chi^{[n-1,1]},\ \chi^{[n-2,2]}$, and $ \chi^{[n-3,3]}$. 
The weighting functions used in \cite{behajaina20203} are $\omega_i$ such that
\begin{align}
\begin{cases}
\omega_1(s,t) &= -\frac{2}{3}t - \frac{2}{3}s  + \frac{1}{3}\alpha +  \frac{2}{3}\beta + \frac{\gamma}{3},\\
\omega_2(s,t) &= \frac{1}{6}t - \frac{1}{3}s  + \frac{1}{6}(\alpha - \beta) - \frac{\gamma}{3},\\
\omega_3(s,t) &= -\frac{1}{2}t + \frac{1}{2}(\alpha - \beta),\\
\omega_4(s,t) &= s,\\
\omega_5(s,t) &= t\label{eq:weights-even}
\end{cases}
\end{align}
where $(t,s)$ belongs to the polytope defined by
\begin{align}
\begin{cases}
&2x + 2y \leq  \binom{n}{3},\\
&x - y \geq 0,\\
& \beta + \gamma - \binom{n-1}{2} \leq x \leq \beta + \gamma + \binom{n-1}{2},\\
&y\geq 0.
\end{cases}
\label{eq:polytope}
\end{align}

It is easy to see that if $n$ is even, then the conjugacy classes with cycle type $(n-5,5), (n-6,2^3), \mbox{ and } (n-6,4,1^2)$ consist of even permutations, whereas the ones with cycle type $(n-6,4,2)$ and $(n-6,5,1)$ are odd permutations. Therefore, we would like the weights in \eqref{eq:weighted_adjacency_matrix_even} to satisfy the following table (the \checkmark means that the weights need not be equal to $0$).
\begin{table}[H]
	\begin{tabular}{|c|c|c|c|c|c|}
		\hline
		Parity of $n$ & $(n-5,5)$ & $(n-6,2^3)$& $(n-6,4,1^2)$& $(n-6,4,2)$& $(n-6,5,1)$ \\ \hline 
		Even & \checkmark & \checkmark & \checkmark & $0$ & $0$ \\ \hline
	\end{tabular}
	\caption{Desired weights for $n$ even}
\end{table}

Let $s_0 = t_0 = 0$. Consider the weights 
\begin{align}
\begin{cases}
\omega_1(s_0,t_0) &= \frac{1}{3}\alpha +  \frac{2}{3}\beta + \frac{\gamma}{3},\\
\omega_2(s_0,t_0) &= \frac{1}{6}(\alpha - \beta) - \frac{\gamma}{3},\\
\omega_3(s_0,t_0) &=  \frac{1}{2}(\alpha - \beta),\\
\omega_4(s_0,t_0) &= 0,\\
\omega_5(s_0,t_0) &= 0.\label{eq:altweights-even}
\end{cases}
\end{align}

Note that in contrast to the previous subsection, $(t_0,s_0)$ does not belong to the polytope given in \eqref{eq:polytope}. 
Fortunately, we can prove that these weights still work. We first prove that the eigenvalues afforded by the irreducible characters of dimension at least $\binom{n}{4}$ is bounded from below by $-1$.

\begin{lem}
	Let $n\geq 8$ and $\lambda \vdash n$. If $f^\lambda > \binom{n}{4}$, then $|\xi_\lambda| < 1$.
\end{lem}
\begin{proof}
	It was proved in \cite[Lemma~6.1]{behajaina20203} that $|\chi^\lambda(x)| \in \{0,1\}$ if $x$ has cycle type $(n-5,5)$ or $(n-6,4,1^2)$. Moreover, $|\chi^\lambda (x)| \in \{0,1,2,3\}$ if $x$ has cycle type $(n-6,2^3)$. Using these information on the character values, we can bound the eigenvalues afforded by irreducible characters of dimension at least $\binom{n}{4}$ as follows.
	\begin{align*}
		|\xi_\lambda| &\leq \frac{1}{f^\lambda} \left( |\omega_{1}| + 3|\omega_{2}| + |\omega_{3}| \right)\\
		&=\frac{1}{f^\lambda} \left( \omega_{1} + 3\omega_{2} + \omega_{3} \right)\\
		&= \frac{1}{f^\lambda} \left( \alpha +2\omega_2 \right)\\
		&<\frac{7 \alpha}{6\binom{n}{4}} < 1, \mbox{ for } n\geq 8 .
	\end{align*}
\end{proof}

\begin{lem}
	If $f^\lambda < \binom{n}{4}$, then $\xi_\lambda \geq -1$.
\end{lem}
\begin{table}[t]
	\centering
	\begin{tabular}{ccccc}
		\hline
		& &&&\\
		& & $(n-5)$ & $(n-6,2^3)$ & $(n-6,4,1^2)$ \\
		Representation & Dimension &&&\\
		\hline
		$\left[n\right]$ & $1$ & $1$ & $1$ & $1$ \\
		$\left[n-1, 1\right]$ & $n-1$ & $-1$ & $-1$ & $1$ \\
		$\left[n-2, 2\right]$ & $\binom{n}{2}-n$ & $0$ & $3$ & $-1$ \\
		$\left[n-2, 1^2\right]$ & $\binom{n-1}{2}$ & $1$ & $-2$ & $0$ \\
		$\left[n-3, 3\right]$ & $\binom{n}{3} - \binom{n}{2}$ & $0$ & $-3$ & $-1$ \\
		$\left[n-3, 2, 1\right]$ & $\frac{n(n-2)(n-4)}{3}$ & $0$ & $0$ & $0$ \\
		$\left[n-3, 1^3\right]$ & $\binom{n-1}{3}$ & $-1$ & $2$ & $0$ \\
		$\left[n-4, 4\right]$ & $\binom{n}{4}-\binom{n}{3}$ & $0$ & $3$ & $1$ \\
		$\left[n-4, 1^4\right]$ & $\binom{n-1}{4}$ & $1$ & $1$ & $-1$ \\
	\end{tabular}
	\caption{Character table for the irreducible characters of low dimension.}\label{tab:3-even}
\end{table}
\begin{proof}
	The irreducible characters of degree less than $\binom{n}{4}$ are given in Lemma~\ref{lem:characters-dim-4}. As explained in the previous section, since the eigenvalues afforded by $\lambda \vdash n$ and its transpose are equal we only need to compute the eigenvalues corresponding to the partitions in the set $S_n$ given in Lemma~\ref{lem:characters-dim-4}. Using  Table~\ref{tab:3-even}, we compute directly the eigenvalues as follows. 
	\begin{align*}
		\xi_{[n]} &= \alpha\\
		\xi_{[n-1,1]} &= \xi_{[n-2,2]} = \xi_{[n-3,3]} = -1\\
		\xi_{[n-2,1^2]} &= \frac{\omega_{1} - 2\omega_{2}}{\binom{n-1}{2}} = \frac{\binom{n}{2}-1}{\binom{n-1}{2}} >0\\
		\xi_{[n-3,2,1]} &= 0\\
		\xi_{[n-3,1^3]} &= \frac{1 - \binom{n}{2}}{\binom{n-1}{3}} > -1\\
		\xi_{[n-4,4]} &= \frac{3\omega_{2}+\omega_{3}}{\binom{n}{4} - \binom{n}{3}} = \frac{\alpha - \beta - \gamma}{\binom{n}{4} - \binom{n}{3}} = \frac{\binom{n}{3} - \binom{n}{2}}{\binom{n}{4} - \binom{n}{3}} >0 \\
		\xi_{[n-4,1^4]} &= \frac{\omega_{1}+\omega_{2} - \omega_{3}}{\binom{n-1}{4}} = \frac{\beta}{\binom{n-1}{4}}>0.
	\end{align*}
	Therefore, all eigenvalues afforded by irreducible characters of dimension less than $\binom{n}{4}$ are bounded from below by $-1$.
\end{proof}
\section{$4$-setwise for $\sym(n)$ and $\alt(n)$}

In this section, we prove Theorem~\ref{thm:main-sym} and Theorem~\ref{thm:main-alt}, for $k = 4$. We will distinguish two cases, depending whether $n$ is even or odd.

Let $\alpha = \binom{n}{4}-1$ and  $\beta, \gamma, \delta,\varepsilon$ be the dimensions of the irreducible characters corresponding to $[n-1,1], [n-2,2], [n-3,3]$ and $[n-4,4]$, respectively. That is, 
\begin{align*}
&\beta = n-1  & \gamma = \binom{n}{2}-n\\
&\delta = \binom{n}{3} - \binom{n}{2} & \varepsilon = \binom{n}{4} - \binom{n}{3}.
\end{align*}

\subsection{Small cases}
When $9 \leq n \leq 21$, we verify with \verb|Sagemath| that there exists a weighted adjacency matrix that satisfies \eqref{first}, \eqref{second} and \eqref{third}. A \verb|Sagemath| code on these eigenvalues is available in \cite{link}. The cases when $n\geq 22$ are considered in the next two sections.

\subsection{Even case}
Let $n\geq 22$ be even. We assign non-zero weights to the conjugacy classes $C_1 :=C_{(n-1,1)}, C_2 :=C_{(n-2,2)}, C_3 :=C_{(n-3,3)}, C_4 :=C_{(n-3,1^3)},$ and $C_5:= C_{(n-7,5,1^2)}$. We consider the weighted adjacency matrix
\begin{align}
	A = \frac{\omega_1}{|C_1|}A_{(n-1,1)} + \frac{\omega_2}{|C_2|} A_{(n-2,2)} + \frac{\omega_3}{|C_3|} A_{(n-3,3)} + \frac{\omega_4}{|C_4|} A_{(n-3,1^3)}+
	\frac{\omega_5}{|C_5|} A_{(n-7,5,1^2)}	,
\end{align}
where $\omega_1,\omega_2,\omega_3,\omega_4$, and $\omega_5$ are to be determined.

We observe that the eigenvalues of $A$ corresponding to the irreducible character $\chi^\lambda$, for $\lambda \vdash n$, is
\begin{align*}
	\xi_{\lambda} &= \frac{1}{f^\lambda} \sum_{i=1}^5 \omega_i \chi^\lambda(C_i).
\end{align*}

We will prove that for some values of $(\omega_i)_{i\in [5]}$, the properties \eqref{first}, \eqref{second}, and \eqref{third} are satisfied. In addition to these eigenvalues, we also want the eigenvalue afforded by $[n-3,2,1]$ to be $-1$. Let $\eta = \frac{n(n-2)(n-4)}{3}$ be the dimension of the irreducible character corresponding to $[n-3,2,1]$. In order to satisfy \eqref{first} and \eqref{second} as well as $\xi_{[n-3,2,1]} = -1$, the following system of linear equations must be satisfied.

\begin{align*}
	\begin{cases}
		\omega_1+ \omega_2 + \omega_3 +\omega_4 +\omega_5 &= \alpha\\
		-\omega_2 -\omega_3+2\omega_4 + \omega_5 &= -\beta\\
		-\omega_1 + \omega_2 -\omega_5 &= -\gamma\\
		-\omega_2 +\omega_3 -2\omega_4 -\omega_5 &= -\delta\\
		 -\omega_3 -\omega_4&= -\varepsilon\\
		 \omega_1 -\omega_3 -\omega_4 &= -\eta.
	\end{cases}
\end{align*}

The above system of linear equations is overdetermined but has a unique solution (note that the fifth equation is the combination of the first four). Its unique solution is  given as follows.

\begin{align}
	\begin{cases}
		\omega_1 &=  \alpha -\beta - \gamma -\delta -\eta = \frac{1}{24} \, n^{4} - \frac{3}{4} \, n^{3} + \frac{71}{24} \, n^{2} - \frac{13}{4} \, n 
		 \\
		\omega_2 &= \frac{1}{2} \left(\beta +\delta \right) = \frac{1}{12} \, n^{3} - \frac{1}{2} \, n^{2} + \frac{11}{12} \, n - \frac{1}{2} 
		, \\
		\omega_3 &= \frac{1}{3} \left( \alpha -\delta +\eta \right) = \frac{1}{72} \, n^{4} - \frac{1}{36} \, n^{3} - \frac{13}{72} \, n^{2} + \frac{19}{36} \, n - \frac{1}{3} 
		, \\
		\omega_4 &= \frac{1}{3}\left( 2\alpha -3\beta -3\gamma -2\delta - \eta \right) = \frac{1}{36} \, n^{4} - \frac{7}{18} \, n^{3} + \frac{41}{36} \, n^{2} - \frac{10}{9} \, n + \frac{1}{3} 
		,\\
		\omega_5 &= \frac{1}{2} \left(-2\alpha +3\beta +4\gamma +3\delta +2\eta \right) = -\frac{1}{24} \, n^{4} + \frac{5}{6} \, n^{3} - \frac{71}{24} \, n^{2} + \frac{8}{3} \, n - \frac{1}{2} 
		.
	\end{cases}\label{eq:weight-4-even}
\end{align}
We observe that
\begin{align*}
	\begin{cases}
		\omega_1, \omega_2, \omega_3,\omega_4\geq 0, \mbox{ for }n\geq 13\\
		\omega_5 <0 \mbox{ for } n\geq 16\\
	\end{cases}
\end{align*}

In the next lemma we determine the character values of $\sym(n)$ on the conjugacy classes $C_i$, for $i\in [5]$.
\begin{lem}
	Let $\lambda \vdash n\geq 22$. For any $i\in \{1,2,3,5\}$ and $x\in C_i$, we have $|\chi^\lambda(x)| \in \{0,1\}$. Moreover, if $x \in C_4$, then $|\chi^\lambda(x)| \in \{0,1,2\}$.\label{lem:char-val-4-even}
\end{lem}
We use  \cite[Lemma~3.4~(3)]{behajaina20203} in our proof. This proof depends on the \emph{Murnaghan-Nakayama Rule} and \emph{rim hooks} (see \cite[Section~4.10]{sagan2001symmetric} for the definition and properties). It was proved in \cite[Lemma~3.4~(3)]{behajaina20203} that if $n\geq 3a+1$, then in any Young diagram $\lambda \vdash n$ there exists at most one rim hook of length $n-a$.
\begin{proof}
	Since $n\geq 22$, there is at most one rim hook of length $n-a$ in any Young diagram with $n$ cells, for any $a\leq 7$. By the Murnaghan-Nakayama rule, we have
	\begin{align*}
		|\chi^\lambda_{(n-1,1)}| & \leq \max_{\mu \vdash 1} |\chi^\mu_{(1)}| = 1\\
		|\chi^\lambda_{(n-2,2)}| & \leq \max_{\mu \vdash 2} |\chi_{(2)}^\mu| = 1 \\
		|\chi^\lambda_{(n-3,3)}| & \leq \max_{\mu \vdash 3} |\chi_{(3)}^\mu| = 1\\
		|\chi^\lambda_{(n-3,1^3)}| & \leq \max_{\mu \vdash 3} |\chi_{(1^3)}^\mu| = f^{[2,1]} = 2 \\
		|\chi^\lambda_{(n-7,5,1^2)}| & \leq \max_{\mu \vdash 7}|\chi^\mu_{(5,1^2)}| =1.
	\end{align*}
	Since the irreducible characters of the symmetric group are integer valued, the proof follows.
\end{proof}

\begin{table}[H]
	\begin{tabular}{ccccccc}
		\hline
		& & & & & & \\
		& & $(n-1, 1)$ & $(n-2, 2)$ & $(n-3, 3)$ & $(n-3, 1^3)$  & $(n-7,5,1^2)$ \\
		Representation &Dimension &  &  &  &  &  \\ \hline
		$\left[n\right]$ &$1$ & $1$ & $1$ & $1$ & $1$ & $1$ \\
		$\left[n-1, 1\right]$ &$n-1$  & $0$ & $-1$ & $-1$ & $2$ & $1$ \\
		$\left[n-2, 2\right]$ &$\binom{n}{2}-n$ & $-1$ & $1$ & $0$ & $0$ & $-1$ \\
		$\left[n-2, 1^2\right]$ & $\binom{n-1}{2}$  & $0$ & $0$ & $1$ & $1$ & $0$ \\
		$\left[n-3, 3\right]$ & $\binom{n}{3}-\binom{n}{2}$ & $0$ & $-1$ & $1$ & $-2$ & $-1$ \\
		$\left[n-3, 2, 1\right]$ & $\frac{n(n-2)(n-4)}{3}$ & $1$ & $0$ & $-1$ & $-1$ & $0$ \\
		$\left[n-3, 1^3\right]$ &$\binom{n-1}{3}$ & $0$ & $0$ & $0$ & $0$ & $0$ \\
		$\left[n-4, 4\right]$ & $\binom{n}{4}-\binom{n}{3}$ & $0$ & $0$ & $-1$ & $-1$ & $0$ \\
		$\left[n-4, 1^4\right]$ & $\binom{n-1}{4}$ & $0$ & $0$ & $0$ & $0$ & $0$ \\
		$\left[n-4, 3, 1\right]$ & $\frac{n(n-1)(n-3)(n-6)}{8}$ & $0$ & $1$ & $0$ & $0$ & $1$ \\
		$\left[n-4, 2^2\right]$ & $\frac{n(n-1)(n-4)(n-5)}{12}$ & $0$ & $-1$ & $1$ & $1$ & $1$ \\
		$\left[n-4, 2, 1^2\right]$ & $\frac{n(n-2)(n-3)(n-5)}{8}$ & $-1$ & $0$ & $0$ & $0$ & $0$ \\
		$\left[n-5, 5\right]$ & $\binom{n}{5} -\binom{n}{4}$ & $0$ & $0$ & $0$ & $0$ & $1$ \\
		$\left[n-5, 1^5\right]$  & $\binom{n-1}{5}$ & $0$ & $0$ & $0$ & $0$ & $1$ \\
	\end{tabular}
	\caption{Character values for $k= 4$ and $n$ even}\label{tab:4-even}
\end{table}

Next, we prove that the irreducible characters of high dimension of $\sym(n)$ must correspond to small eigenvalues of $A$.
\begin{lem}
	If $\lambda \vdash n$ such that $f^\lambda > \binom{n}{5}$, then $|\xi_\lambda |< 1$.\label{lem:high-dim-4-odd}
\end{lem}
\begin{proof}
	Let $\lambda \vdash n$ such that $f^\lambda > \binom{n}{5}$. The eigenvalue of $A$ corresponding to $\chi^\lambda$ is 
	\begin{align*}
		\xi_{\lambda} &= \sum_{i=1}^5 \omega_i \chi^\lambda(C_i).
	\end{align*}
	By the triangle inequality, we have
	\begin{align*}
		|\xi_\lambda| &\leq \frac{1}{f^\lambda} \left(|\omega_1| + |\omega_2| +|\omega_3| + 2|\omega_4| + |\omega_5|\right)\\
		& = \frac{1}{f^\lambda}\left(\omega_1 + \omega_2 + \omega_3 +2\omega_4 -\omega_5\right)
		\\
		&= \frac{1}{f^\lambda} \left( \alpha +\omega_4 - 2\omega_5 \right)\\
		&< \frac{\frac{11}{72} \, n^{4} - \frac{83}{36} \, n^{3} + \frac{541}{72} \, n^{2} - \frac{241}{36} \, n + \frac{1}{3} }{\binom{n}{5}}<1  \hspace*{5cm}  (\mbox{ for }n\geq 3 ).
	\end{align*}
	Therefore, $|\xi_{\lambda}| <1.$
\end{proof}

Now, we deal with the irreducible characters of dimension at most $\binom{n}{5}$. The partitions affording these characters are given in Lemma~\ref{lem:characters-dim-4}.

\begin{lem}
	If $\lambda \vdash n$ such that $f^\lambda < \binom{n}{5}$, then $-1\leq \xi_{\lambda}\leq \alpha$.\label{lem:low-dim-4-odd}
\end{lem}
\begin{proof}
	By Lemma~\ref{lem:characters-dim-4}, we know all the partitions whose corresponding irreducible characters are of degree at most $\binom{n}{5}$. Using Table~\ref{tab:4-even}, we can compute the eigenvalues corresponding to these partitions. They are as follows.
	\begin{align*}
		\xi_{[n]} &= \alpha \\
		\xi_{[n-1,1]} &= -1 \\
		\xi_{[n-2,2]} &= -1\\
		\xi_{[n-2,1^2]} &= \frac{\omega_3 +\omega_4}{\binom{n-1}{2}}=  \frac{\varepsilon}{\binom{n-1}{2}} >0  \\
		\xi_{[n-3,3]} &= -1\\
		\xi_{[n-3,2,1]} &= -1 \hspace*{1cm} \\
		\xi_{[n-3,1^3]} &= 0 \\
		\xi_{[n-4,4]} &= -1\\
		\xi_{[n-4,1^4]} &= 0 \\
		\xi_{[n-4,3,1]} &= \frac{8(\omega_2 + \omega_5)}{n(n-1)(n-3)(n-6)} = -\tfrac{n^{4} - 22 \, n^{3} + 83 \, n^{2} - 86 \, n + 24}{3 \, {\left(n^{4} - 10 \, n^{3} + 27 \, n^{2} - 18 \, n\right)}}>-1, \hspace{0.1cm} (\mbox{ for }n\geq 7) \\
		\xi_{[n-4,2^2]} &= \frac{12(-\omega_2 +\omega_3 +\omega_4 +\omega_5)}{n(n-1)(n-4)(n-5)} = \tfrac{12(\gamma +\eta)}{n(n-1)(n-4)(n-5)} >0, \hspace{1cm} (\mbox{ for }n\geq 4) \\
		\xi_{[n-4,2,1^2]} &= \frac{-8\omega_1}{n(n-2)(n-3)(n-5)} = -\tfrac{n^{4} - 18 \, n^{3} + 71 \, n^{2} - 78 \, n}{3 \, {\left(n^{4} - 10 \, n^{3} + 31 \, n^{2} - 30 \, n\right)}} >-1, \hspace{0.5cm} (\mbox{ for } n\geq 4) \\
		\xi_{[n-5,5]} &= \frac{\frac{1}{2} \left(-2\alpha +3\beta +4\gamma +3\delta +2\eta \right)}{\binom{n}{5}-\binom{n}{4}} = -\tfrac{5 \, {\left(n^{4} - 20 \, n^{3} + 71 \, n^{2} - 64 \, n + 12\right)}}{n^{5} - 15 \, n^{4} + 65 \, n^{3} - 105 \, n^{2} + 54 \, n} >-1, \ \ (\mbox{ for }n\geq 10) 
		\\
		\xi_{[n-5,1^5]} &= \frac{\omega_5}{\binom{n-1}{5}} = \frac{\frac{1}{2} \left(-2\alpha +3\beta +4\gamma +3\delta +2\eta \right)}{\binom{n-1}{5}} = -\tfrac{5 \, {\left(n^{4} - 20 \, n^{3} + 71 \, n^{2} - 64 \, n + 12\right)}}{n^{5} - 15 \, n^{4} + 85 \, n^{3} - 225 \, n^{2} + 274 \, n - 120} >-1,
	\end{align*}
	$\mbox{ for }n\geq 16$.
\end{proof}

Combining Lemma~\ref{lem:high-dim-4-odd} and Lemma~\ref{lem:low-dim-4-odd}, we conclude that \eqref{first}, \eqref{second} and \eqref{third} are satisfied by the weighted adjacency matrix $A$. Therefore, Theorem~\ref{thm:main-sym} and Theorem~\ref{thm:main-alt} hold for $k = 4$.

\subsection{Odd case}
The approach used here is almost similar to the even case. Let $n\geq 23$ be odd. We will need to make the eigenvalues afforded by the irreducible characters corresponding to $[n-3,2,1], [n-2,1^2]$ and $[n-3,1^3]$ be equal to $-1$.

We assign non-zero weights to the conjugacy classes $C_1 :=C_{(n)}, C_2 :=C_{(n-2,1^2)}, C_3 :=C_{(n-3,2,1)},\\ C_4 :=C_{(n-6,3^2)},C_5:= C_{(n-7,6,1)}$ and $C_6 := C_{(n-9,6,1^3)}$. We consider the following weighted adjacency matrix of $\Gamma_{n,k}$.

\begin{align}
	\begin{split}
		A & = \frac{\omega_1}{|C_1|}A_{(n)} + \frac{\omega_2}{|C_2|} A_{(n-2,1^2)} + \frac{\omega_3}{|C_3|} A_{(n-3,2,1)} + \frac{\omega_4}{|C_4|}  A_{(n-6,3^2)} +\frac{\omega_5}{|C_5|}A_{(n-7,6,1)} \\ &\hspace{0.5cm}+\frac{\omega_6}{|C_6|} A_{(n-9,6,1^3)}.
	\end{split}
\end{align}

In addition to $\alpha,\beta,\gamma,\delta,$ and $\varepsilon$ which were defined in the beginning of this section, we let
\begin{align*}
	\zeta & = \frac{n(n-2)(n-4)}{3},\\
	\eta &= \binom{n-1}{2},\\
	\theta &= \binom{n-1}{3}. 
\end{align*}
The numbers $\zeta,\eta,$ and $\theta$ are respectively the dimensions of the irreducible characters afforded by the partitions $[n-3,2,1], [n-2,1^2],$ and $[n-3,1^3]$, respectively. In order for \eqref{first} and \eqref{second} to be satisfied, as well as the property explained in the beginning of this subsection, the weights $(\omega_i)_{i\in [6]}$ need to be as follows.

\begin{align*}
	\begin{cases}
		\omega_1 + \omega_2 + \omega_3 + \omega_4 + \omega_5 + \omega_6 &= \alpha,\\
		-\omega_1+ \omega_2 -\omega_4 +2\omega_6 &= -\beta,\\
		-\omega_2 - \omega_5 &= - \gamma,\\
		-\omega_2 +2\omega_4 -2\omega_6 &= - \delta,\\
		-\omega_3 -2\omega_4 -\omega_6 &= -\varepsilon,\\
		\omega_3 -2\omega_4 + \omega_5 -\omega_6 &= -\zeta,\\
		\omega_1 -\omega_3 +\omega_4 +\omega_6 &= -\eta,\\
		-\omega_1 + \omega_4 &=-\theta.
	\end{cases}
\end{align*}

We note that $-\varepsilon = \beta + \gamma +\delta -\alpha$ and $\beta +\delta = \theta$. The system of linear equation is overdetermined, but has a unique solution (note that the fifth equation is the sum of the first fourth equations and the last equation is the sum of the second and the fourth equations). The weights $(\omega_i)_{i=1,2,3,4,5,6}$ are given as follows.

\begin{align}
	\begin{cases}
		\omega_1 &= \frac{1}{6} \left( \alpha +5\beta +4\delta +\zeta \right) = \frac{1}{144} \, n^{4} + \frac{1}{8} \, n^{3} - \frac{133}{144} \, n^{2} + \frac{43}{24} \, n - 1,  \\
		\omega_2 &= \beta +\gamma +\delta +\zeta +\eta = \frac{1}{2} \, n^{3} - 2 \, n^{2} + \frac{3}{2} \, n, \\
		\omega_3 &= \frac{1}{2}\left( \alpha -\gamma +\eta \right) = \frac{1}{48} \, n^{4} - \frac{1}{8} \, n^{3} + \frac{11}{48} \, n^{2} - \frac{1}{8} \, n
		, \\
		\omega_4 &= \frac{1}{6} \left( \alpha -\beta -2\delta +\zeta \right) = \frac{1}{144} \, n^{4} - \frac{1}{24} \, n^{3} + \frac{11}{144} \, n^{2} - \frac{1}{24} \, n
		,\\
		\omega_5 &= -\beta -\delta -\zeta - \eta = -\frac{1}{2} \, n^{3} + \frac{5}{2} \, n^{2} - 3 \, n
		, 
		\\
		\omega_6 &= \frac{1}{6} \left( \alpha -4\beta -3\gamma -2\delta -2\zeta -3\eta \right) = \frac{1}{144} \, n^{4} - \frac{5}{24} \, n^{3} + \frac{83}{144} \, n^{2} - \frac{3}{8} \, n
		.
	\end{cases}\label{eq:weight-4-odd}
\end{align}
We observe the following.
\begin{align}
	\begin{cases}
		\omega_1,\omega_2, \omega_3 , \omega_4,\omega_6 >0, \mbox{ for }n\geq 27\\
		\omega_5 <0, \mbox{ for } n\geq 3.
	\end{cases}\label{eq:weights-sign-4-odd}
\end{align}

\begin{lem}
	Let $\lambda \vdash n \geq 28$. For any $x \in C_1\cup C_2\cup C_3 \cup C_5$, $|\chi^\lambda(x)| \in \{0,1\}$. Moreover, if $x\in C_4 \cup C_6$, then $|\chi^\lambda(x)| \in \{0,1,2\}$.\label{lem:char-val-4-odd}
\end{lem}
\begin{proof}
	The proof is similar to the proof of Lemma~\ref{lem:char-val-4-even}. By \cite[Lemma~3.4]{behajaina20203}, there exists a unique rim hook of lenght $n-a$ in the Young diagram corresponding to $\lambda$, for any $a\in \{0,2,3,6,7,9\}$. By the Murnaghan-Nakayama rule, we have
	\begin{align*}
		|\chi^\lambda_{(n)}| &\leq 1\\
		|\chi^\lambda_{(n-2,1^2)}| &\leq \max_{\mu \vdash 2} |\chi^{\mu}_{(1^2)}| = 1\\
		|\chi^\lambda_{(n-3,2,1)}| &\leq \max_{\mu \vdash 3} |\chi^\mu_{(2,1)}| = 1\\
		|\chi^\lambda_{(n-6,3^2)}| &\leq \max_{\mu \vdash 6} |\chi^\mu_{(3^2)}| =  2\\
		|\chi^\lambda_{(n-7,6,1)}| &\leq \max_{\mu \vdash 7} |\chi^\mu_{(6,1)}| = 1\\
		|\chi^\lambda_{(n-9,6,1^3)}| & \leq \max_{\mu \vdash 9} |\chi^\lambda_{(6,1^3)}| = 2.
	\end{align*}
\end{proof}

\begin{table}[H]
	\begin{tabular}{ccccccc}
		\hline
		& & & & & & \\
		& $(n)$ & $(n-2,1^2)$ & $(n-3,2,1)$ & $(n-6,3^2)$ & $(n-7,6,1)$ & $(n-9,6,1^3)$ \\
		Representation&  & & & & & \\ \hline
		$\left[n\right]$ & $1$ & $1$ & $1$ & $1$ & $1$ & $1$ \\
		$\left[n-1, 1\right]$ & $-1$ & $1$ & $0$ & $-1$ & $0$ & $2$ \\
		$\left[n-2, 2\right]$ & $0$ & $-1$ & $0$ & $0$ & $-1$ & $0$ \\
		$\left[n-2, 1^2\right]$ & $1$ & $0$ & $-1$ & $1$ & $0$ & $1$ \\
		$\left[n-3, 3\right]$ & $0$ & $-1$ & $0$ & $2$ & $0$ & $-2$ \\
		$\left[n-3, 2, 1\right]$ & $0$ & $0$ & $1$ & $-2$ & $1$ & $-1$ \\
		$\left[n-3, 1^3\right]$ & $-1$ & $0$ & $0$ & $1$ & $0$ & $0$ \\
		$\left[n-4, 4\right]$ & $0$ & $0$ & $-1$ & $-2$ & $0$ & $-1$ \\
		$\left[n-4, 1^4\right]$ & $1$ & $0$ & $0$ & $-1$ & $0$ & $0$ \\
		$\left[n-4, 3, 1\right]$ & $0$ & $1$ & $0$ & $0$ & $0$ & $0$ \\
		$\left[n-4, 2^2\right]$ & $0$ & $1$ & $-1$ & $2$ & $0$ & $1$ \\
		$\left[n-4, 2, 1^2\right]$ & $0$ & $0$ & $0$ & $0$ & $-1$ & $0$ \\
		$\left[n-5, 5\right]$ & $0$ & $0$ & $0$ & $0$ & $0$ & $0$ \\
		$\left[n-5, 1^5\right]$ & $-1$ & $0$ & $0$ & $1$ & $0$ & $0$ \\
	\end{tabular}
	\caption{Character table for low dimensional characters when $n$ is odd.}\label{tab:4-odd}
\end{table}

We use the previous lemma to prove that the eigenvalues corresponding to the high-dimensional irreducible characters must be at least $-1$ and small in absolute values.

\begin{lem}
	If $f^\lambda > \binom{n}{5}$, then $|\xi_{\lambda}| < 1$.
\end{lem}
\begin{proof}
	By the triangle inequality, we have
	\begin{align*}
		|\xi_{\lambda}| &\leq \frac{1}{f^\lambda}\left(|\omega_1| + |\omega_2| + |\omega_3| + 2|\omega_4|+|\omega_5| + 2|\omega_6|\right)\\
		&\leq \frac{1}{f^\lambda}\left(\omega_1 +\omega_2 + \omega_3 +2\omega_4 - \omega_5 + 2\omega_6\right)\\
		&= \frac{1}{f^\lambda}\left(\alpha +\omega_{4} -2\omega_5 + \omega_{6}\right)\\
		&= \frac{1}{f^\lambda}\left( \frac{23}{144} \, n^{4} - \frac{181}{72} \, n^{3} + \frac{1165}{144} \, n^{2} - \frac{509}{72} \, n + \frac{1}{3}		
		 \right)\\
		&< \frac{\frac{23}{144} \, n^{4} - \frac{181}{72} \, n^{3} + \frac{1165}{144} \, n^{2} - \frac{509}{72} \, n + \frac{1}{3}
		}{\binom{n}{5}} <1, \hspace{1cm}\mbox{ for }  n\geq 3.
	\end{align*}
\end{proof}

Next, we compute the eigenvalues afforded by the low dimensional characters.
\begin{lem}
	Let $\lambda \vdash n$. If $f^\lambda <\binom{n}{5}$, then $\xi_{\lambda}\geq -1$.
\end{lem}
\begin{proof}
	The irreducible characters of dimension less than $\binom{n}{5}$ are given in Lemma~\ref{lem:characters-dim-4}. First, we have
	\begin{align*}
		\xi_{[1^n]} &= \xi{[n]} = \alpha \\
		\xi_{[n-1,1]} &= \xi_{[n-2,2]} = \xi_{[n-3,3]} = \xi_{[n-4,4]} = \xi_{[n-2,1^2]}= \xi_{[n-3,2,1]}= \xi_{[n-3,1^3]} = -1.
	\end{align*}
	The remaining eigenvalues are
	\begin{align*}
		\xi_{[n-4,1^4]} &= \tfrac{\beta + \delta}{\binom{n-1}{{4}}} = \frac{4 \, {\left(n^{3} - 6 \, n^{2} + 11 \, n - 6\right)}}{n^{4} - 10 \, n^{3} + 35 \, n^{2} - 50 \, n + 24} >0,\\
		\xi_{[n-4,3,1]} &= \frac{8\omega_2}{n(n-1)(n-3)(n-6)} = \frac{4 \, {\left(n^{3} - 4 \, n^{2} + 3 \, n\right)}}{n^{4} - 10 \, n^{3} + 27 \, n^{2} - 18 \, n} >0,\\
		\xi_{[n-4,2^2]} & = \frac{12\left(\omega_2 -\omega_{3} + 2\omega_{4} + \omega_{6}\right)}{n(n-1)(n-4)(n-5)} = \frac{2 \, {\left(2 \, n^{3} - 9 \, n^{2} + 7 \, n\right)}}{n^{4} - 10 \, n^{3} + 29 \, n^{2} - 20 \, n} >0
		,\\
		\xi_{[n-4,2,1^2]} &= \frac{-8\omega_5}{n(n-2)(n-3)(n-5)} = \frac{4 \, {\left(n^{3} - 5 \, n^{2} + 6 \, n\right)}}{n^{4} - 10 \, n^{3} + 31 \, n^{2} - 30 \, n} >0,\\
		\xi_{[n-5,5]} &= 0,\\
		\xi_{[n-5,1^5]} &= \frac{-(\beta +\delta)}{\binom{n-1}{5}} = -\frac{20 \, {\left(n^{3} - 6 \, n^{2} + 11 \, n - 6\right)}}{n^{5} - 15 \, n^{4} + 85 \, n^{3} - 225 \, n^{2} + 274 \, n - 120} >-1.
	\end{align*}
	Therefore, the eigenvalues of $A$ afforded by $\lambda \vdash n$ with $f^\lambda <\binom{n}{5}$ are all at least $-1$.
\end{proof}

From the previous two lemmas, the weighted adjacency matrix $A$  satisfies \eqref{first}, \eqref{second}, and \eqref{third}. We conclude from the ratio bound that Theorem~\ref{thm:main-sym} and Theorem~\ref{thm:main-alt} hold for $k=4$ and $n\geq 28$ (see \eqref{eq:weights-sign-4-odd} for the lower bound on $n$). 

\section{$5$-setwise for $\sym(n)$ and $\alt(n)$}

In this section we use a similar argument to the previous section to prove Theorem~\ref{thm:main-sym} and Theorem~\ref{thm:main-alt} for $k = 5$. First, we consider the cases when $n$ is small. Then, we distinguish the cases whether $n$ is even or odd. We define
\begin{align*}
	\begin{aligned}
		\alpha &= \binom{n}{5}-1 &\beta = n-1 \qquad\qquad & \gamma = \binom{n}{2}-n\\
		\delta &=\binom{n}{3} - \binom{n}{2} &\varepsilon = \binom{n}{4}-\binom{n}{3} \qquad &\iota = \binom{n}{5}-\binom{n}{4}\\
		\eta &= \binom{n-1}{2} &\zeta=\frac{n(n-2)(n-4)}{3} \qquad\qquad &\theta = \frac{n(n-1)(n-4)(n-5)}{12}\\
		\nu &= \frac{n(n-1)(n-3)(n-6)}{8}&\tau = \binom{n-1}{4} \qquad  &\kappa = \frac{n(n-2)(n-3)(n-5)}{8}.
	\end{aligned}
\end{align*}

\subsection{Small cases}
When $11 \leq n \leq 30$, we verified using \verb|Sagemath| that there exists a weighted adjacency matrix that satisfies \eqref{first}, \eqref{second}, and \eqref{third} (see \cite{link}). The cases when $n\geq 32$ even and $n\geq 31$ odd are considered in the next two sections.
\subsection{Even case}
Let $n\geq 32$ be even and let
\begin{align*}
	\begin{aligned}
	 	C_1 &:= C_{(n-1,1)}\qquad &C_2:= C_{(n-2,2)} \qquad &C_3 := C_{(n-3,1^3)} \\
	 	C_4 &:= C_{(n-4,4)} \qquad
	 	&C_5 := C_{(n-4,2,1^2)}\qquad
	 	&C_6 := C_{(n-6,6)}\\
	 	C_7 &:= C_{(n-6,2^3)} \qquad
	 	&C_8 := C_{(n-7,3^2,1)} \qquad
	 	&C_9 := C_{(n-8,6,1^2)} \\
	 	& \qquad &C_{10} := C_{(n-10,6,1^4)} &. 
	\end{aligned}
\end{align*}

Consider the weighted adjacency matrix of $\Gamma_{n,5}$ given by
\begin{align}
	\begin{split}
	A & = \frac{\omega_1}{|C_1|}A_{(n-1,1)} + \frac{\omega_2}{|C_2|}A_{(n-2,2)} + \frac{\omega_3}{|C_3|} A_{(n-3,1^3)} + \frac{\omega_4}{|C_4|}A_{(n-4,4)} + \frac{\omega_5}{|C_5|} A_{(n-4,2,1^2)} \\& \hspace{0.5cm} + \frac{\omega_6}{|C_6|} A_{(n-6,6)} + \frac{\omega_7
	}{|C_7|} A_{(n-6,2^3)} + \frac{\omega_8}{|C_8|} A_{(n-7,3^2,1)} + \frac{\omega_9}{|C_9|} A_{(n-8,6,1^2)}  \\& \hspace{0.5cm}+ \frac{\omega_{10}}{|C_{10}|}A_{(n-10,6,1^4)}.
	\end{split}\label{eq:weight-5-even}
\end{align}

\begin{table}[b]
	\centering
	\tiny 
	\begin{tabular}{ccccccccccc}
		\hline
		&&&&&&&&&&\\
		Representation & $(n-1,1)$ & $(n-2,2)$ & $(n-3,1^3)$ & $(n-4,4)$ & $(n-4,2,1^2)$ & $(n-6,6)$ & $(n-6,2^3)$ & $(n-7,3^2,1)$ & $(n-8,6,1^2)$ & $(n-10,6,1^4)$ \\
		&&&&&&&&&&\\
		\hline 
		$\left[n\right]$ & $1$ & $1$ & $1$ & $1$ & $1$ & $1$ & $1$ & $1$ & $1$ & $1$ \\
		$\left[n-1, 1\right]$ & $0$ & $-1$ & $2$ & $-1$ & $1$ & $-1$ & $-1$ & $0$ & $1$ & $3$ \\
		$\left[n-2, 2\right]$ & $-1$ & $1$ & $0$ & $0$ & $0$ & $0$ & $3$ & $-1$ & $-1$ & $2$ \\
		$\left[n-2, 1^2\right]$ & $0$ & $0$ & $1$ & $1$ & $-1$ & $1$ & $-2$ & $0$ & $0$ & $3$ \\
		$\left[n-3, 3\right]$ & $0$ & $-1$ & $-2$ & $0$ & $0$ & $0$ & $-3$ & $2$ & $-1$ & $-2$ \\
		$\left[n-3, 2, 1\right]$ & $1$ & $0$ & $-1$ & $0$ & $0$ & $0$ & $0$ & $-1$ & $0$ & $0$ \\
		$\left[n-3, 1^3\right]$ & $0$ & $0$ & $0$ & $-1$ & $-1$ & $-1$ & $2$ & $2$ & $0$ & $1$ \\
		$\left[n-4, 4\right]$ & $0$ & $0$ & $-1$ & $1$ & $-1$ & $0$ & $3$ & $0$ & $0$ & $-3$ \\
		$\left[n-4, 1^4\right]$ & $0$ & $0$ & $0$ & $0$ & $0$ & $1$ & $1$ & $0$ & $0$ & $0$ \\
		$\left[n-4, 3, 1\right]$ & $0$ & $1$ & $0$ & $-1$ & $1$ & $0$ & $0$ & $0$ & $1$ & $-3$ \\
		$\left[n-4, 2^2\right]$ & $0$ & $-1$ & $1$ & $0$ & $0$ & $0$ & $3$ & $0$ & $1$ & $0$ \\
		$\left[n-4, 2, 1^2\right]$ & $-1$ & $0$ & $0$ & $1$ & $1$ & $0$ & $-3$ & $-1$ & $0$ & $-1$ \\
		$\left[n-5, 5\right]$ & $0$ & $0$ & $0$ & $-1$ & $-1$ & $0$ & $-3$ & $-2$ & $0$ & $-1$ \\
		$\left[n-5, 1^5\right]$ & $0$ & $0$ & $0$ & $0$ & $0$ & $-1$ & $-1$ & $0$ & $0$ & $0$ \\
		$\left[n-5, 4, 1\right]$ & $0$ & $0$ & $1$ & $0$ & $0$ & $0$ & $0$ & $0$ & $0$ & $0$ \\
		$\left[n-5, 3, 2\right]$ & $0$ & $0$ & $2$ & $1$ & $-1$ & $0$ & $-3$ & $0$ & $0$ & $3$ \\
		$\left[n-5, 3, 1^2\right]$ & $0$ & $-1$ & $0$ & $0$ & $0$ & $0$ & $3$ & $0$ & $-1$ & $0$ \\
		$\left[n-5, 2^2, 1\right]$ & $0$ & $1$ & $0$ & $-1$ & $-1$ & $0$ & $0$ & $2$ & $-1$ & $1$ \\
		$\left[n-5, 2, 1^3\right]$ & $1$ & $0$ & $0$ & $0$ & $0$ & $0$ & $0$ & $-1$ & $0$ & $0$ \\
		$\left[n-6, 6\right]$ & $0$ & $0$ & $0$ & $0$ & $0$ & $1$ & $1$ & $1$ & $1$ & $1$ \\
		$\left[n-6, 1^6\right]$ & $0$ & $0$ & $0$ & $0$ & $0$ & $0$ & $0$ & $1$ & $-1$ & $-1$ \\
	\end{tabular}	
	\caption{Character table for the low dimensional character when $n$ is even}\label{tab:5-even}
\end{table}

We would like \eqref{first}, \eqref{second}, and \eqref{third} be satisfied. In addition, we also want the eigenvalues corresponding to the partitions $[n-2,1^2], [n-3,2,1]$, $[n-4,3,1]$, $[n-4,2^2]$, $[n-4,2,1^2]$, $[n-4,1^4]$ and $[n-3,1^3]$ be equal to $-1$. The dimension of the irreducible characters corresponding to these seven partitions are $\eta$, $\zeta,$ $\nu$, $\theta$, $\kappa$, $\tau$, and $\mu$, respectively. Therefore, a system of linear equations with thirteen equations and ten variables must be satisfied. We can reduce this system of linear equations to correspond to the following matrix equation.

\begin{align}
	\begin{bmatrix}
	1 & 1 & 1 & 1 & 1 & 1 & 1 & 1 & 1 & 1 \\
	0 & -1 & 2 & -1 & 1 & -1 & -1 & 0 & 1 & 3 \\
	-1 & 1 & 0 & 0 & 0 & 0 & 3 & -1 & -1 & 2 \\
	0 & -1 & -2 & 0 & 0 & 0 & -3 & 2 & -1 & -2 \\
	0 & 0 & -1 & 1 & -1 & 0 & 3 & 0 & 0 & -3 \\
	0 & 0 & 1 & 1 & -1 & 1 & -2 & 0 & 0 & 3 \\
	1 & 0 & -1 & 0 & 0 & 0 & 0 & -1 & 0 & 0 \\
	0 & -1 & 1 & 0 & 0 & 0 & 3 & 0 & 1 & 0 \\
	-1 & 0 & 0 & 1 & 1 & 0 & -3 & -1 & 0 & -1 \\
	0 & 0 & 0 & -1 & -1 & -1 & 2 & 2 & 0 & 1
	\end{bmatrix}
	 \begin{bmatrix}
	\omega_{1}\\
	\omega_{2}\\
	\omega_{3}\\
	\omega_{4}\\
	\omega_{5}\\
	\omega_{6}\\
	\omega_{7}\\
	\omega_{8}\\
	\omega_{9}\\
	\omega_{10}
	\end{bmatrix}=
	\begin{bmatrix}
	\alpha\\
	-\beta\\
	-\gamma\\
	-\delta\\
	-\varepsilon\\
	-\eta \\
	-\zeta\\
	-\theta\\
	-\kappa\\
	- \mu
	\end{bmatrix}\label{eq:system-5-even}
\end{align}

The unique solution to the system of linear equation in \eqref{eq:system-5-even} is
\begin{align}
	\begin{cases}
		\omega_{1} &= \frac{1}{6} \, \alpha- \beta- \frac{1}{6} \, \delta- \varepsilon+ \frac{1}{6} \, \zeta+ \theta+ \kappa+ \frac{5}{6} \, \mu\\
		\omega_{2} &= \frac{1}{8} \, \alpha+ \frac{3}{8} \, \beta+ \frac{1}{8} \, \gamma+ \frac{3}{8} \, \delta+ \frac{1}{4} \, \varepsilon+ \frac{1}{8} \, \eta+ \frac{1}{4} \, \theta- \frac{1}{4} \, \kappa- \frac{3}{8} \, \mu\\
		\omega_{3} &= -\beta- \varepsilon+ \zeta+ \theta+ \kappa+ \mu\\
		\omega_{4} &= \frac{1}{4} \, \alpha- \frac{3}{4} \, \beta- \frac{1}{4} \, \gamma- \frac{1}{4} \, \delta- \varepsilon- \frac{1}{4} \, \eta+ \frac{1}{2} \, \theta+ \frac{1}{4} \, \mu\\
		\omega_{5} &= \frac{1}{4} \, \alpha- \frac{1}{4} \, \beta- \frac{1}{4} \, \gamma- \frac{1}{4} \, \delta+ \frac{1}{4} \, \eta+ \frac{1}{4} \, \mu\\
		\omega_{6} &= -\frac{1}{24} \, \alpha+ \frac{25}{24} \, \beta+ \frac{1}{8} \, \gamma+ \frac{1}{24} \, \delta+ \frac{13}{12} \, \varepsilon- \frac{1}{24} \, \eta- \frac{11}{12} \, \theta- \frac{1}{12} \, \kappa- \frac{1}{24} \, \mu\\
		\omega_{7} &= \frac{1}{24} \, \alpha- \frac{1}{24} \, \beta- \frac{1}{8} \, \gamma- \frac{1}{24} \, \delta- \frac{1}{12} \, \varepsilon+ \frac{1}{24} \, \eta- \frac{1}{12} \, \theta+ \frac{1}{12} \, \kappa+ \frac{1}{24} \, \mu\\
		\omega_{8} &= \frac{1}{6} \, \alpha- \frac{1}{6} \, \delta+ \frac{1}{6} \, \zeta- \frac{1}{6} \, \mu\\
		\omega_{9} &= \frac{3}{2} \, \beta+ \frac{1}{2} \, \gamma+ \frac{1}{2} \, \delta+ \frac{3}{2} \, \varepsilon- \zeta- \frac{3}{2} \, \theta- \frac{3}{2} \, \kappa- \frac{3}{2} \, \mu\\
		\omega_{10} &= \frac{1}{24} \, \alpha+ \frac{1}{8} \, \beta- \frac{1}{8} \, \gamma- \frac{1}{24} \, \delta+ \frac{1}{4} \, \varepsilon- \frac{1}{8} \, \eta- \frac{1}{3} \, \zeta- \frac{1}{4} \, \theta- \frac{1}{4} \, \kappa- \frac{7}{24} \, \mu.
	\end{cases}\label{eq:weights-5-even}
\end{align}

The weights in \eqref{eq:weights-5-even} are as follows when expressed in terms of $n$.

\begin{align}
	\begin{cases}
		\omega_{1} &= \frac{1}{720} \, n^{5} + \frac{11}{72} \, n^{4} - \frac{209}{144} \, n^{3} + \frac{307}{72} \, n^{2} - \frac{119}{30} \, n \\
		\omega_{2} &= \frac{1}{960} \, n^{5} - \frac{1}{96} \, n^{4} + \frac{7}{192} \, n^{3} - \frac{5}{96} \, n^{2} + \frac{1}{40} \, n \\
		\omega_{3} &= \frac{1}{6} \, n^{4} - \frac{7}{6} \, n^{3} + \frac{7}{3} \, n^{2} - \frac{4}{3} \, n \\
		\omega_{4} &= \frac{1}{480} \, n^{5} - \frac{1}{48} \, n^{4} + \frac{7}{96} \, n^{3} - \frac{5}{48} \, n^{2} + \frac{1}{20} \, n \\
		\omega_{5} &= \frac{1}{480} \, n^{5} - \frac{1}{48} \, n^{4} + \frac{7}{96} \, n^{3} - \frac{5}{48} \, n^{2} + \frac{1}{20} \, n \\
		\omega_{6} &= -\frac{1}{2880} \, n^{5} - \frac{11}{288} \, n^{4} + \frac{233}{576} \, n^{3} - \frac{415}{288} \, n^{2} + \frac{83}{40} \, n - 1 \\
		\omega_{7} &= \frac{1}{2880} \, n^{5} - \frac{1}{288} \, n^{4} + \frac{7}{576} \, n^{3} - \frac{5}{288} \, n^{2} + \frac{1}{120} \, n \\
		\omega_{8} &= \frac{1}{720} \, n^{5} - \frac{1}{72} \, n^{4} + \frac{7}{144} \, n^{3} - \frac{5}{72} \, n^{2} + \frac{1}{30} \, n \\
		\omega_{9} &= -\frac{1}{4} \, n^{4} + 2 \, n^{3} - \frac{19}{4} \, n^{2} + 3 \, n \\
		\omega_{10} &= \frac{1}{2880} \, n^{5} - \frac{13}{288} \, n^{4} + \frac{151}{576} \, n^{3} - \frac{137}{288} \, n^{2} + \frac{31}{120} \, n.\label{eq:weight-5-even-function-of-n}
	\end{cases}
\end{align}

The weights in \eqref{eq:weight-5-even} are only monotone when $n$ is large enough. Therefore, considering the sign of the weights as in previous sections is not ideal. However, for $n\geq 30$, $|\omega_i| < \frac{1}{3}\alpha$, for all $i\in \{2,4,5,6,7,8,10\}$. Moreover, $|\omega_{i}|<\frac{10}{9}\alpha$ for $i \in \{1,3,9\}$ and $n\geq 30$.

Next, we give a lemma about the character values on the conjugacy classes used in our weighted adjacency matrix. The proof is omitted since it is similar to the proofs of Lemma~\ref{lem:char-val-4-even} and Lemma~\ref{lem:char-val-4-odd}. 
\begin{lem}
	Let $\lambda \vdash n \geq 31$. For any $x\in C_1\cup C_2 \cup C_4 \cup C_5 \cup C_6 \cup C_9$, $|\chi^\lambda(x)| \in \{0,1\}$. Moreover, for $x\in C_3 \cup C_8$, $|\chi^\lambda(x)| \in \{0,1,2\}$ and for $x \in C_7 \cup C_{10}$, $|\chi^\lambda(x)| \in \{0,1,2,3\}$.
\end{lem}

Now, we use these weights on the matrix in \eqref{eq:weight-5-even-function-of-n}. The eigenvalues of $A$ afforded by irreducible characters of dimension at least $\binom{n}{6}$ are again bounded by $1$, in absolute value.

\begin{lem}
	If $\lambda \vdash n$ such that $f^\lambda > 2\binom{n}{6}$, then $|\xi_\lambda|< 1$.
\end{lem}
\begin{proof}
	Let $\lambda \vdash n$ such that $f^\lambda > 2\binom{n}{6}$. Using the fact that all the weights are at most $\frac{1}{2}\alpha$ in absolute value, we have
	\begin{align*}
		|\xi_\lambda| &\leq \frac{1}{f^\lambda}(|\omega_1|+|\omega_2|+2|\omega_3|+|\omega_4| + |\omega_5|+ |\omega_6|+ 3|\omega_7|+ 2|\omega_8| + |\omega_{9}| +3 |\omega_{10}|)\\
		&<  \frac{1}{f^\lambda}\left(4\times \frac{10}{9} + 12 \times \frac{1}{3}\right)\alpha\\
		&<\frac{76\alpha}{9\times 2\binom{n}{6}} <1, \hspace{1cm} \mbox{ for }n\geq 31.
	\end{align*}
\end{proof}

Next, we prove that the eigenvalues afforded by the irreducible characters of dimension less than $2\binom{n}{6}$ are at least $-1$.

\begin{lem}
	If $\lambda \vdash n$ such that $f^\lambda< 2\binom{n}{6}$, then $ \xi_\lambda \geq -1$.
\end{lem}
\begin{proof}
	It is enough to compute the eigenvalues corresponding to the irreducible characters in Lemma~\ref{lem:characters-dim-6}, whenever $n\geq 27$. We group the eigenvalues depending on whether the irreducible characters have dimension less than $\binom{n}{4}$, between $\binom{n}{4}$ and $\binom{n}{5}$, or between $\binom{n}{5}$ and $2\binom{n}{6}$.
	
	\noindent If $f^\lambda < \binom{n}{4}$, then
	\begin{align*}
		\xi_{[n]} &= \alpha,\\
		\xi_{[n-1,1]} &= \xi_{[n-2,2]} = \xi_{[n-3,3]} = \xi_{[n-4,4]} =\xi_{[n-2,1^2]} = \xi_{[n-3,2,1]} = \xi_{[n-4,1^4]}\xi_{[n-3,1^3]}  =-1.
	\end{align*}
	If $\binom{n}{4}<f^\lambda <\binom{n}{5}$, then 
	\begin{align*}
		\xi_{[n-4,2,1^2]} &= -1,\\
		\xi_{[n-4,2^2]} &= -1 ,\\
		\xi_{[n-4,3,1]} &= -1,\\
		\xi_{[n-5,1^5]} &= \tfrac{-\omega_{6}- \omega_7}{\binom{n-1}{5}} = \tfrac{5 \, {\left(n^{4} - 10 \, n^{3} + 35 \, n^{2} - 50 \, n + 24\right)}}{n^{5} - 15 \, n^{4} + 85 \, n^{3} - 225 \, n^{2} + 274 \, n - 120}
		 >0, \mbox{ for }n\geq 6,\\
		\xi_{[n-5,5]} &= -1.
	\end{align*}
	If $\binom{n}{5} < f^\lambda <2\binom{n}{6}$, then the corresponding eigenvalues are
	\begin{align*}
		\xi_{[n-5,4,1]} &= \tfrac{30\omega_3}{n(n-1)(n-2)(n-4)(n-8)} = \tfrac{5 \, {\left(n^{4} - 7 \, n^{3} + 14 \, n^{2} - 8 \, n\right)}}{n^{5} - 15 \, n^{4} + 70 \, n^{3} - 120 \, n^{2} + 64 \, n}
		 >-1 ,\\
		\xi_{[n-5,3,1^2]} &= \tfrac{20 \left( -\omega_2 +3\omega_7 -\omega_9 \right)}{n(n-1)(n-3)(n-4)(n-7)} = \tfrac{5 \, {\left(n^{4} - 8 \, n^{3} + 19 \, n^{2} - 12 \, n\right)}}{n^{5} - 15 \, n^{4} + 75 \, n^{3} - 145 \, n^{2} + 84 \, n} >-1,\\
		\xi_{[n-5,3,2]} &= \tfrac{24\left( 2\omega_3 + \omega_4 - \omega_5 -3\omega_7 +3\omega_{10} \right)}{n(n-1)(n-2)(n-5)(n-7)} = \tfrac{5 \, n^{4} - 38 \, n^{3} + 79 \, n^{2} - 46 \, n}{n^{5} - 15 \, n^{4} + 73 \, n^{3} - 129 \, n^{2} + 70 \, n} >-1,\\
		\xi_{[n-5,2^2,1]} &= \tfrac{24\left( \omega_2 -\omega_{4} -  \omega_5 +2\omega_8 -\omega_{9} +\omega_{10}\right)}{n(n-1)(n-3)(n-5)(n-6)} = \tfrac{5 \, n^{4} - 42 \, n^{3} + 103 \, n^{2} - 66 \, n}{n^{5} - 15 \, n^{4} + 77 \, n^{3} - 153 \, n^{2} + 90 \, n} >-1,\\
		\xi_{[n-5,2,1^3]} &= \tfrac{30 (\omega_1-\omega_{8})}{n(n-2)(n-3)(n-4)(n-6)} = \tfrac{5 \, {\left(n^{4} - 9 \, n^{3} + 26 \, n^{2} - 24 \, n\right)}}{n^{5} - 15 \, n^{4} + 80 \, n^{3} - 180 \, n^{2} + 144 \, n} >-1,\\
		\xi_{[n-6,1^6]} &= \tfrac{\omega_8 -\omega_9-\omega_{10}}{\binom{n-1}{6}} = \tfrac{3 \, {\left(n^{5} + 270 \, n^{4} - 2125 \, n^{3} + 4950 \, n^{2} - 3096 \, n\right)}}{4 \, {\left(n^{6} - 21 \, n^{5} + 175 \, n^{4} - 735 \, n^{3} + 1624 \, n^{2} - 1764 \, n + 720\right)}}>-1,\\
		\xi_{[n-6,6]} &= \tfrac{\omega_{6} + \omega_7 + \omega_8 + \omega_9 +\omega_{10}}{\binom{n}{6}-\binom{n}{5}} = \tfrac{5 \, {\left(n^{5} - 202 \, n^{4} + 1571 \, n^{3} - 3890 \, n^{2} + 3096 \, n - 576\right)}}{4 \, {\left(n^{6} - 21 \, n^{5} + 145 \, n^{4} - 435 \, n^{3} + 574 \, n^{2} - 264 \, n\right)}} >-1, \mbox{ for }n\geq 22.
	\end{align*}
This completes the proof.
\end{proof}

\subsubsection*{Summary}
In this subsection, we have proved that for $n \geq 32$ even, all eigenvalues of the weighted adjacency matrix $A$ given in \eqref{eq:weighted_adjacency_matrix_even} are in the interval $\left[ -1,\alpha \right]$. Therefore, \eqref{first}, \eqref{second}, and \eqref{third} are satisfied for $n\geq 32$ even.

\subsection{Odd case}
Now, we assume that $n\geq 31$ is odd. Let $C_1 := C_{(n)},\  C_2:=C_{(n-2,1^2)},\ C_3 := C_{(n-3,2,1)},\ C_4 := C_{(n-4,2^2)},\ C_5 := C_{(n-4,1^4)},\ C_6 := C_{(n-4,3,1)},C_7 := C_{(n-7,6,1)},\ C_8 := C_{(n-8,4^2)}$, and $C_9 := C_{(n-9,6,1^3)}$. We consider the weighted adjacency matrix 
\begin{align}
	\begin{split}
		A &= \frac{\omega_1}{|C_1|} A_{(n)} + \frac{\omega_2}{|C_2|} A_{(n-2,1^2)}+\frac{\omega_3}{|C_3|} A_{(n-3,2,1)} +\frac{\omega_4}{|C_4|} A_{(n-4,2^2)} + \frac{\omega_5}{|C_5|} A_{(n-4,1^4)}\\ &\hspace{0.5cm} +\frac{\omega_6}{|C_6|} A_{(n-4,3,1)} +\frac{\omega_7}{|C_7|}A_{(n-7,6,1)} + \frac{\omega_8}{|C_8|}A_{(n-8,4^2)}+ \frac{\omega_9}{|C_9|}A_{(n-9,6,1^3)}.
	\end{split}\label{eq:weighted-adj-5-odd}
\end{align}

Again, we would like to find the weights $(\omega_i)_{i\in [9]}$ so that \eqref{first}, \eqref{second}, and \eqref{third} are satisfied. In addition, we also would like the eigenvalues afforded by $[n-2,1^2]$, $[n-3,2,1]$, $[n-4,3,1]$, $[n-4,2^2]$, $[n-4,2,1^2]$, and $[n-4,1^4]$ to be $-1$. The dimensions of these six irreducible characters are $\eta = \binom{n-1}{2}$, $\zeta = \frac{n(n-2)(n-4)}{3}$, $\nu = \frac{n(n-1)(n-3)(n-6)}{8}$, $\theta = \frac{n(n-1)(n-4)(n-5)}{12}$, $\kappa = \frac{n(n-2)(n-3)(n-5)}{8}$, and $\tau = \binom{n-1}{4}$. Using Table~\ref{tab:5-odd}, the following system of linear equations must be satisfied.

\begin{align*}
	\begin{cases}
		\omega_{1} + \omega_{2} + \omega_{3} + \omega_{4} + \omega_{5} + \omega_{6} + \omega_{7} + \omega_{8} + \omega_{9} &= \alpha,\\
		-\omega_{1} + \omega_{2} - \omega_{4} + 3 \, \omega_{5} - \omega_{8} + 2 \, \omega_{9} &= -\beta,\\
		-\omega_{2} + 2 \, \omega_{4} + 2 \, \omega_{5} - \omega_{6} - \omega_{7} &= -\gamma,\\
		-\omega_{2} - 2 \, \omega_{4} - 2 \, \omega_{5} + \omega_{6} - 2 \, \omega_{9} &= -\delta,\\
		-\omega_{3} + \omega_{4} - 3 \, \omega_{5} + 2 \, \omega_{8} - \omega_{9} &= -\varepsilon,\\
		-\omega_{4} - \omega_{5} - \omega_{6} - 2 \, \omega_{8} &= -\iota,\\
		\omega_{1} - \omega_{3} - \omega_{4} + 3 \, \omega_{5} + \omega_{8} + \omega_{9} &= -\eta,\\
		\omega_{3} + \omega_{7} - \omega_{9}  &= -\zeta,\\
		\omega_{2} + \omega_{4} - 3 \, \omega_{5} - 2 \, \omega_{8} &= -\nu,\\
		\omega_{2} - \omega_{3} + \omega_{9} &= -\theta,\\
		-\omega_{4} - \omega_{5} - \omega_{6} - \omega_{7} + 2 \, \omega_{8} &= -\kappa,\\
		\omega_{1} - \omega_{8} &= -\tau,\\
	\end{cases}
\end{align*}

\begin{table}[H]
	\tiny
	\begin{tabular}{cccccccccc}
		\hline
		&&&&&&&&&\\
		& $(n)$ & $(n-2,1^2)$ & $(n-3,2,1)$ & $(n-4,2^2)$ & $(n-4,1^4)$ & $(n-4,3,1)$ & $(n-7,6,1)$ & $(n-8,4^2)$ & $(n-9,6,1^3)$ \\
		Representation &&&&&&&&&\\ \hline
		$\left[n\right]$ & $1$ & $1$ & $1$ & $1$ & $1$ & $1$ & $1$ & $1$ & $1$ \\
		$\left[n-1, 1\right]$ & $-1$ & $1$ & $0$ & $-1$ & $3$ & $0$ & $0$ & $-1$ & $2$ \\
		$\left[n-2, 2\right]$ & $0$ & $-1$ & $0$ & $2$ & $2$ & $-1$ & $-1$ & $0$ & $0$ \\
		$\left[n-2, 1^2\right]$ & $1$ & $0$ & $-1$ & $-1$ & $3$ & $0$ & $0$ & $1$ & $1$ \\
		$\left[n-3, 3\right]$ & $0$ & $-1$ & $0$ & $-2$ & $-2$ & $1$ & $0$ & $0$ & $-2$ \\
		$\left[n-3, 2, 1\right]$ & $0$ & $0$ & $1$ & $0$ & $0$ & $0$ & $1$ & $0$ & $-1$ \\
		$\left[n-3, 1^3\right]$ & $-1$ & $0$ & $0$ & $1$ & $1$ & $1$ & $0$ & $-1$ & $0$ \\
		$\left[n-4, 4\right]$ & $0$ & $0$ & $-1$ & $1$ & $-3$ & $0$ & $0$ & $2$ & $-1$ \\
		$\left[n-4, 1^4\right]$ & $1$ & $0$ & $0$ & $0$ & $0$ & $0$ & $0$ & $-1$ & $0$ \\
		$\left[n-4, 3, 1\right]$ & $0$ & $1$ & $0$ & $1$ & $-3$ & $0$ & $0$ & $-2$ & $0$ \\
		$\left[n-4, 2^2\right]$ & $0$ & $1$ & $-1$ & $0$ & $0$ & $0$ & $0$ & $0$ & $1$ \\
		$\left[n-4, 2, 1^2\right]$ & $0$ & $0$ & $0$ & $-1$ & $-1$ & $-1$ & $-1$ & $2$ & $0$ \\
		$\left[n-5, 5\right]$ & $0$ & $0$ & $0$ & $-1$ & $-1$ & $-1$ & $0$ & $-2$ & $0$ \\
		$\left[n-5, 1^5\right]$ & $-1$ & $0$ & $0$ & $0$ & $0$ & $0$ & $0$ & $1$ & $0$ \\
		$\left[n-5, 4, 1\right]$ & $0$ & $0$ & $1$ & $0$ & $0$ & $0$ & $0$ & $0$ & $1$ \\
		$\left[n-5, 3, 2\right]$ & $0$ & $0$ & $0$ & $-1$ & $3$ & $0$ & $0$ & $2$ & $2$ \\
		$\left[n-5, 3, 1^2\right]$ & $0$ & $-1$ & $0$ & $0$ & $0$ & $0$ & $0$ & $0$ & $0$ \\
		$\left[n-5, 2^2, 1\right]$ & $0$ & $-1$ & $0$ & $1$ & $1$ & $1$ & $0$ & $-2$ & $0$ \\
		$\left[n-5, 2, 1^3\right]$ & $0$ & $0$ & $0$ & $0$ & $0$ & $0$ & $1$ & $0$ & $0$ \\
		$\left[n-6, 6\right]$ & $0$ & $0$ & $0$ & $0$ & $0$ & $0$ & $1$ & $0$ & $1$ \\
		$\left[n-6, 1^6\right]$ & $1$ & $0$ & $0$ & $0$ & $0$ & $0$ & $-1$ & $-1$ & $-1$ \\
	\end{tabular}
	\caption{$n$ odd.}\label{tab:5-odd}
\end{table}

The above system of linear equations is overdetermined, however, it has a unique solution. We note that the sixth equation is a linear combination of the first five and that the seventh equation is a combination of the 9th, 10th, and 12th equations (note that $\nu = \tau +\theta -\eta $). The unique solution to this system of linear equations is given by

\begin{align*}
	\begin{cases}
		\omega_1 &=\frac{1}{8} \, \alpha+ \frac{7}{8} \, \beta+ \frac{3}{4} \, \varepsilon+ \frac{1}{8} \, \eta- \frac{7}{8} \, \theta- \frac{1}{8} \, \kappa\\
		\omega_2 &= \frac{1}{2} \, \alpha- \frac{1}{2} \, \beta- \gamma- \delta- \zeta- \frac{1}{2} \, \eta- \frac{3}{2} \, \theta- \frac{1}{2} \, \kappa\\
		\omega_3 &= \frac{1}{4} \, \alpha- \frac{1}{4} \, \beta- \frac{1}{4} \, \eta+ \frac{3}{4} \, \theta- \frac{1}{4} \, \kappa\\
		\omega_4 &= \frac{1}{8} \, \alpha- \frac{1}{8} \, \beta- \frac{1}{4} \, \gamma- \frac{1}{4} \, \varepsilon- \frac{1}{8} \, \eta+ \frac{1}{8} \, \theta+ \frac{1}{8} \, \kappa\\
		\omega_5 &= \frac{1}{8} \, \alpha- \frac{1}{8} \, \beta- \frac{5}{12} \, \gamma- \frac{1}{3} \, \delta+ \frac{1}{12} \, \varepsilon- \frac{1}{3} \, \zeta+ \frac{1}{24} \, \eta- \frac{13}{24} \, \theta- \frac{1}{24} \, \kappa\\
		\omega_6 &= \frac{1}{2} \, \alpha- \frac{1}{2} \, \beta- \frac{1}{3} \, \gamma- \frac{2}{3} \, \delta- \frac{1}{3} \, \varepsilon+ \frac{1}{3} \, \zeta- \frac{1}{6} \, \eta+ \frac{1}{6} \, \theta+ \frac{1}{6} \, \kappa\\
		\omega_7 &= -\frac{1}{2} \, \alpha+ \frac{1}{2} \, \beta+ \gamma+ \delta+ \frac{1}{2} \, \eta+ \frac{1}{2} \, \theta+ \frac{1}{2} \, \kappa\\
		\omega_8 &= \frac{1}{8} \, \alpha- \frac{1}{8} \, \beta- \frac{1}{4} \, \varepsilon+ \frac{1}{8} \, \eta+ \frac{1}{8} \, \theta- \frac{1}{8} \, \kappa\\
		\omega_9 &= -\frac{1}{4} \, \alpha+ \frac{1}{4} \, \beta+ \gamma+ \delta+ \zeta+ \frac{1}{4} \, \eta+ \frac{5}{4} \, \theta+ \frac{1}{4} \, \kappa\\
	\end{cases}
\end{align*}

The expressions of these weights in terms of $n$ are as follows.
\begin{align*}
	\begin{cases}
		\omega_1 &= \frac{1}{960} \, n^{5} - \frac{5}{96} \, n^{4} + \frac{29}{64} \, n^{3} - \frac{145}{96} \, n^{2} + \frac{253}{120} \, n - 1,\\
		\omega_2 &= \frac{1}{240} \, n^{5} - \frac{7}{24} \, n^{4} + \frac{103}{48} \, n^{3} - \frac{119}{24} \, n^{2} + \frac{31}{10} \, n,\\
		\omega_3 &= \frac{1}{480} \, n^{5} - \frac{1}{48} \, n^{4} + \frac{7}{96} \, n^{3} - \frac{5}{48} \, n^{2} + \frac{1}{20} \, n,\\
		\omega_4 &= \frac{1}{960} \, n^{5} - \frac{1}{96} \, n^{4} + \frac{7}{192} \, n^{3} - \frac{5}{96} \, n^{2} + \frac{1}{40} \, n,\\
		\omega_5 &= \frac{1}{960} \, n^{5} - \frac{5}{96} \, n^{4} + \frac{55}{192} \, n^{3} - \frac{49}{96} \, n^{2} + \frac{11}{40} \, n,\\
		\omega_6 &= \frac{1}{240} \, n^{5} - \frac{1}{24} \, n^{4} + \frac{7}{48} \, n^{3} - \frac{5}{24} \, n^{2} + \frac{1}{10} \, n,\\
		\omega_7 &= -\frac{1}{240} \, n^{5} + \frac{5}{24} \, n^{4} - \frac{79}{48} \, n^{3} + \frac{109}{24} \, n^{2} - \frac{41}{10} \, n,\\
		\omega_8 &= \frac{1}{960} \, n^{5} - \frac{1}{96} \, n^{4} + \frac{7}{192} \, n^{3} - \frac{5}{96} \, n^{2} + \frac{1}{40} \, n,\\
		\omega_9 &= -\frac{1}{480} \, n^{5} + \frac{3}{16} \, n^{4} - \frac{119}{96} \, n^{3} + \frac{39}{16} \, n^{2} - \frac{83}{60} \, n.
	\end{cases}
\end{align*}

Now, we present a lemma about the character values on the conjugacy classes that we chose. We omit the proof since it is similar to the proofs of Lemma~\ref{lem:char-val-4-even} and Lemma~\ref{lem:char-val-4-odd}. 
\begin{lem}
	Let $\lambda \vdash n \geq 28$. For any $x\in C_1\cup C_2 \cup C_3\cup C_6 \cup C_7$, $|\chi^\lambda(x)| \in \{0,1\}$. Moreover, for $x\in C_4 \cup C_8 \cup C_9$, $|\chi^\lambda(x)| \in \{0,1,2\}$ and for $x \in C_5$, $|\chi^\lambda(x)| \in \{0,1,2,3\}$.
\end{lem}

Similar to the previous subsection, it is not hard to verify that $|\omega_i| < \frac{3}{5}\alpha, $ for any $i\in [9]$ and $n\geq 31$. Using this fact, we are able to prove the following lemma. 
\begin{lem}
	If $\lambda \vdash n$ such that $f^\lambda \geq 2\binom{n}{6}$, then $|\xi_\lambda| <1$.
\end{lem}  
\begin{proof}
	The eigenvalue corresponding to $\lambda$ is such that 
	\begin{align*}
		|\xi_\lambda| &\leq  \frac{1}{f^\lambda} \left( |\omega_1| + |\omega_2| + |\omega_3| +2 |\omega_4| + 3|\omega_5| + |\omega_6| +|\omega_7| + 2|\omega_8|+2 |\omega_9| \right)\\
		&\leq  \frac{14}{f^\lambda} \times \frac{3\alpha}{5} \\
		&<\frac{42\alpha}{10\binom{n}{6}} <1, \hspace{1cm} \mbox{ for }n\geq 31.
	\end{align*}
	This completes the proof.
\end{proof}

Next, we show that all eigenvalues of $A$ afforded by the irreducible characters of degree less than $2\binom{n}{6}$ are at least $-1$.
\begin{lem}
	Let $\lambda \vdash n$. If $f^\lambda < 2\binom{n}{6}$, then $\xi_\lambda \geq -1$.
\end{lem}
\begin{proof}
	Since the matrix $A$ in \eqref{eq:weighted-adj-5-odd} satisfies \eqref{first} and \eqref{second}, we have 
	\begin{align*}
		\xi_{[n]} &= \alpha\\
		\xi_{[n-1,1]}&= \xi_{[n-2,2]} = \xi_{[n-3,3]} = \xi_{[n-4,4]} = \xi_{[n-5,5]} = -1.
	\end{align*}
	If $f^\lambda < \binom{n}{4}$, then 
	\begin{align*}
		\xi_{[n-2,1^2]} &= -1 \\
		\xi_{[n-3,2,1]} &= -1\\
		\xi_{[n-3,1^3]} &= \tfrac{-\omega_1+\omega_4 +\omega_5 +\omega_6-\omega_8}{\binom{n-1}{3}}=\tfrac{n^{5} - 10 \, n^{4} - 5 \, n^{3} + 190 \, n^{2} - 416 \, n + 240}{40 \, {\left(n^{3} - 6 \, n^{2} + 11 \, n - 6\right)}}>-1, \mbox{ for }n\geq 4\\
		\xi_{[n-4,1^4]} &= -1.
	\end{align*}
	If $\binom{n}{4} < f^\lambda < \binom{n}{5}$, then 
	\begin{align*}
		\xi_{[n-4,3,1]} &= -1\\
		\xi_{[n-4,2^2]} &= -1 \\
		\xi_{[n-4,2,1^2]} &= -1\\
		\xi_{[n-5,1^5]} &= \frac{-\omega_1 +\omega_8}{\binom{n-1}{5}} =\tfrac{5 \, {\left(n^{4} - 10 \, n^{3} + 35 \, n^{2} - 50 \, n + 24\right)}}{n^{5} - 15 \, n^{4} + 85 \, n^{3} - 225 \, n^{2} + 274 \, n - 120} >0.
	\end{align*}
	If $\binom{n}{5} < f^\lambda < 2\binom{n}{6}$, then 
	\begin{align*}
		\xi_{[n-5,4,1]} &= \tfrac{30\left(\omega_{3} +\omega_{9}\right)}{n(n-1)(n-2)(n-4)(n-8)} = 
		\tfrac{5 \, {\left(n^{4} - 7 \, n^{3} + 14 \, n^{2} - 8 \, n\right)}}{n^{5} - 15 \, n^{4} + 70 \, n^{3} - 120 \, n^{2} + 64 \, n} >0 , \mbox{ for } n\geq 9\\
		\xi_{[n-5,3,2]} &= \tfrac{24\left( -\omega_{4} + 3\omega_{5} +2\omega_{8} + 2\omega_{9} \right)}{n(n-1)(n-2)(n-5)(n-7)} = \tfrac{5 \, n^{4} - 38 \, n^{3} + 79 \, n^{2} - 46 \, n}{n^{5} - 15 \, n^{4} + 73 \, n^{3} - 129 \, n^{2} + 70 \, n}		
		 >0, \mbox{ for } n\geq 8\\
		\xi_{[n-5,3,1^2]} &= \tfrac{-20\omega_2}{n(n-1)(n-3)(n-4)(n-7)} = -\tfrac{n^{5} - 70 \, n^{4} + 515 \, n^{3} - 1190 \, n^{2} + 744 \, n}{12 \, {\left(n^{5} - 15 \, n^{4} + 75 \, n^{3} - 145 \, n^{2} + 84 \, n\right)}}
		 >-1, \mbox{ for }n\geq 8\\
		\xi_{[n-5,2^2,1]} & = \tfrac{24\left( -\omega_{2} + \omega_{4} + \omega_{5} + \omega_{6} -2 \omega_{8} \right)}{n(n-1)(n-3)(n-5)(n-6)} = \tfrac{5 \, n^{4} - 42 \, n^{3} + 103 \, n^{2} - 66 \, n}{n^{5} - 15 \, n^{4} + 77 \, n^{3} - 153 \, n^{2} + 90 \, n}
		 >0 ,\mbox{ for }n\geq 7 \\
		\xi_{[n-5,2,1^3]} &= \tfrac{30\omega_{7}}{n(n-2)(n-3)(n-4)(n-6)} = -\tfrac{n^{5} - 50 \, n^{4} + 395 \, n^{3} - 1090 \, n^{2} + 984 \, n}{8 \, {\left(n^{5} - 15 \, n^{4} + 80 \, n^{3} - 180 \, n^{2} + 144 \, n\right)}}
		 >-1, \mbox{ for } n\geq 7\\
		\xi_{[n-6,6]} &= \tfrac{\omega_7 +\omega_{9}}{\binom{n}{6}-\binom{n}{5}} =  -\tfrac{3 \, {\left(3 \, n^{5} - 190 \, n^{4} + 1385 \, n^{3} - 3350 \, n^{2} + 2632 \, n\right)}}{2 \, {\left(n^{6} - 21 \, n^{5} + 145 \, n^{4} - 435 \, n^{3} + 574 \, n^{2} - 264 \, n\right)}}
		>-1, \mbox{ for } n\geq 11\\
		\xi_{[n-6,1^6]} &= \tfrac{\omega_1 - \omega_{7} - \omega_{8} - \omega_{9}}{\binom{n-1}{6}} = \tfrac{3 \, {\left(3 \, n^{5} - 210 \, n^{4} + 1585 \, n^{3} - 4050 \, n^{2} + 3632 \, n - 480\right)}}{2 \, {\left(n^{6} - 21 \, n^{5} + 175 \, n^{4} - 735 \, n^{3} + 1624 \, n^{2} - 1764 \, n + 720\right)}}
		>-1 , \mbox{ for } n\geq 20 .
	\end{align*}
	This completes the proof.
\end{proof}

\subsubsection*{Summary} In this subsection, we proved that for $n\geq 31$ odd there exists a weighted adjacency matrix of a certain spanning subgraph of $\Gamma_{n,5}$ for which \eqref{first}, \eqref{second}, and \eqref{third} hold.

\section{Applications}\label{sect:application}
In this section, we apply the result in Theorem~\ref{thm:main-sym} and Theorem~\ref{thm:main-alt} to prove some results on the intersection density of Kneser graphs. 

	Let $G\leq \sym(\Omega)$ and let $\Omega_k$ be the collection of all $k$-subsets of $\Omega$. 
It is well-known that $G$ acts on $\Omega_k$. If the latter is transitive, then we say that $G$ is 
$k$-\emph{homogeneous}. We say that $G$ is $k$\emph{-transitive} if $G$ acts transitively on the set of all pairwise distinct entries $k$-tuples of $\Omega$.  We recall the following well-known results on $k$-homogeneous subgroups.

\begin{thm}[Theorem~9.4B \cite{dixon1996permutation}]
	Suppose that $G$ is $k$-homogeneous on $\Omega$ of size $n\geq 2k$. Then $G$ is $(k-1)$-transitive. Moreover, $G$ is $k$-transitive with the exception of
	\begin{itemize}
		\item $k=2$, $\operatorname{ASL}_1(q) \leq G \leq \operatorname{A\Sigma L}_1(q)$, with $q \equiv 3 (\operatorname{mod}4)$
		\item $k = 3$ and $\psl{2}{q} \leq G \leq \operatorname{P\Sigma L}(2,q)$, where $n-1 = q \equiv 3\ (\operatorname{mod} 4)$
		\item $k = 3$, $G \in \{ \agl{1}{8},\agammal{1}{8},\agammal{1}{32}  \}$ 
		\item  $k = 4$ and $G \in \left\{ \psl{2}{8},\ \pgammal{2}{8},\ \pgammal{2}{32} \right\}$.
	\end{itemize}\label{thm:classification-homogeneous}
\end{thm}

\begin{thm}
	Let $G\leq \sym(\Omega)$ be $k$-transitive. If the pair $(G,\Omega)$ is not equal to $\alt(\Omega)$ with $|\Omega| \geq 5$ nor $\sym(\Omega)$ with $|\Omega|\geq 7$, then $G$ is one of the following.
	\begin{itemize}
		\item $k=5$ and $G$ is one of the Mathieu groups $\mathieu{12}$ or $\mathieu{24}$.
		\item $k =4$ and $G$ is one the Mathieu groups $\mathieu{11},\mathieu{12}, \mathieu{23}$ or $\mathieu{24}$.
		\item $k=3$ and $G$ is one of the above or $G$ is $\asl{d}{2}$, $V_{16}.\alt(7)$ (of degree $16$), $\mathieu{11}$ (of degree $12$), $\mathieu{22}, \Aut(\mathieu{22})$ or 
		\begin{align*}
		\psl{2}{q} \leq G \leq \pgammal{2}{q},
		\end{align*}
		of degree $q+1$, where $q$ is a prime power.
	\end{itemize}\label{thm:classification-transitive}
\end{thm}

Let $G$ be a transitive subgroup of automorphism of $K(n,k)$. Then, $G$ is $k$-homogeneous. Assume that $G$ is not equal to the symmetric group or the alternating group. By Theorem~\ref{thm:classification-homogeneous} and Theorem~\ref{thm:classification-transitive}, a $4$-homogeneous group is $4$-transitive or is one of $\psl{2}{8}$ (of degree $9$), $\pgammal{2}{8}$ (of degree $9$) or $\pgammal{2}{32}$ (of degree $33$). Moreover, a $5$-homogeneous group is $5$-transitive and is one of $\mathieu{12}$ (of degree $12$) or $\mathieu{24}$ (of degree $24$). Using \verb|Sagemath|, we obtained the intersection density of some of these $4$-homogeneous groups in the next table.
	
	\begin{table}[H]
		\begin{tabular}{|c|c|c|c|}
			\hline
			Groups & $\psl{2}{8}$ & $\pgammal{2}{8}$& $\mathieu{11}$ \\ \hline
			Intersection density & $2$ & $1$  & $3$ \\
			\hline 
		\end{tabular}
	\caption{Intersection density of some of the $4$-homogeneous groups.}
	\end{table}

	For the group $G = \pgammal{2}{32}$ acting on the $4$-subsets of $[33]$, the intersection density can be computed by considering the maximum cliques in the complement of the derangement graph $\Gamma_G$. The largest and least eigenvalues of $\overline{\Gamma_G}$ are $1023$ and $-33$, respectively. Using the Ratio bound, we have $\alpha(\overline{\Gamma_G})) \leq \frac{|G|}{1-\frac{1023}{-33}} = 5115.$ By the well-known clique-coclique bound \cite[Section~3.7]{godsil2016algebraic}, we have that $\alpha(\Gamma_G) = \omega(\overline{\Gamma_G}) \leq \frac{|G|}{5115} = 32$. Using a computer search, there are cliques of size $32$ in $\overline{\Gamma_G}$. Therefore, $\alpha(\Gamma_G) = 32$ and $\rho(G) =\frac{32}{4} = 8$. 
	
	We were not able to find the exact intersection density of  the four remaining $4$-homogeneous groups. Instead, we found the largest intersecting sets which are subgroups in these groups. We report these numbers in the next table. 
	\begin{table}[H]
		\begin{tabular}{|c|c|c|c|}
			\hline
			Groups &  $\mathieu{12}$& $\mathieu{23}$ & $\mathieu{24}$\\ \hline
			Size of the largest intersecting subgroups & $192$ & $1152$  & $23040$ \\
			\hline
		\end{tabular}
		\caption{Lower bound on the intersection density of the $4$-homogeneous Mathieu groups.}
	\end{table}

	For the $5$-homogeneous group $\mathieu{12}$, we computed using \verb|Sagemath| that its intersection density is equal to $3$. Due to its size and its degree, we were not able to find the intersection density of the Mathieu group $\mathieu{24}$ acting on $5$-subsets of $[23].$ By a computer search, we were able to find intersecting sets of size $5760$ in this group.

	We obtain the following theorem as a corollary of the main results of this paper.
	\begin{thm}
		Let $n$ and $k$ be two positive integers such that $n\geq 2k+1$. The intersection density of the Kneser graph $K(n,k)$, where $k\in \{4,5\}$, is as follows.
		\begin{align*}
		\rho(K(n,4))
		&=
		\begin{cases}
		1 & \mbox{ when } n \not\in \{9,11,12,23,24\}\\
		2 & \mbox{ when } n  = 9\\
		3 & \mbox{ when } n =11\\
		8 & \mbox{ when } n = 33\\
		\end{cases}
		\end{align*}
		and 
		\begin{align*}
		\rho(K(n,5))
		&=
		\begin{cases}
		1 & \mbox{ when }n \not\in \{12,24\}\\
		3  & \mbox{ when }n = 12.\\
		\end{cases}
		\end{align*}\label{thm:Kneser-graph}
	\end{thm}

\section{Future work}
In this paper, we proved that if $G$ is the permutation group $\alt(n)$ or $\sym(n)$ in their actions on the $k$-subsets of $[n]$ with $k\in \{4,5\}$, then $\rho(G) = 1$. We applied this to obtain almost all the intersection density of the Kneser graphs $K(n,4)$ and $K(n,5)$.  To prove these results, we used a method that was considered in \cite{meagher2021erdHos} to find a nonasymptotic proof of \cite[Conjecture~1]{ellis2012setwise} (for $k = 2$). Our main tools are the representation theory of $\sym(n)$ and the weighted ratio bound. Even though we have refined the method in \cite{meagher2021erdHos} in this paper to work for the alternating group (see Section~\ref{sect:idea}), we believe that this method has reached its limit.

In Theorem~\ref{thm:Kneser-graph}, we could not find the intersection density of the groups $\mathieu{12},\mathieu{23}$, and $\mathieu{24}$ via \verb|Sagemath| due to the order of these groups. A similar situation happened to $\mathieu{23}$ acting on the $5$-subsets. We end this paper by posing some open problems.

\begin{prob}
	Find the intersection density of the groups $\mathieu{12}$ and $\mathieu{24}$, in their actions o the $4$-subsets. Find the intersection density of $\mathieu{23}$ in its action on the $4$-subsets and $5$-subsets.
\end{prob}

\begin{prob}
	Find the intersection density of the Kneser graphs $K(n,2)$, for $n\geq 5$, and $K(n,3)$ for $n\geq 7$.\label{prob:1}
\end{prob}
Since a $2$-transitive group is also $2$-homogeneous, finding the intersection density of $K(n,2)$ will require studying the intersection density of the action of the $2$-transitive groups on $2$-subsets. To find the intersection density of the graph $K(n,3)$, one needs to find the intersection density of the groups in Theorem~\ref{thm:classification-homogeneous} and Theorem~\ref{thm:classification-transitive} for $k=3$. 
Therefore, we expect Problem~\ref{prob:1} to be extremely difficult, yet interesting. A paper by Karen Meagher and the third author about the intersection density of the groups in Theorem~\ref{thm:classification-homogeneous} and Theorem~\ref{thm:classification-transitive}, for $k=3$, is currently in preparation.

\vspace*{0.5cm}
\noindent\textsc{Acknowledgement:} We are grateful to Karen Meagher for finding the intersection density of the group $\pgammal{2}{32}$ and finding the lower bound on the intersection density of the $4$-homogeneous Mathieu groups $\mathieu{12},\ \mathieu{23}$, and $\mathieu{24}$ in their actions on the $4$-subsets (and the action on the $5$-subsets for $\mathieu{24}$).

\appendix
\section{Low dimensional of irreducible characters}

\begin{landscape}
	\pagestyle{empty}
\begin{table}[ht]
	\caption{}
	\centering
	\small 
	\begin{tabular}{|c|c|c|c|c|}
		\hline
		Character $\phi$ & Degree $D$ of $\phi$ &$\binom{n}{4}-D$ &$\binom{n}{5}-D$& $2\binom{n}{6}-D$\\		
		\hline
		$[n]$& $1$  & $\frac{(n-4)(n+1)(n^2-3n+6)}{24}$ & $\frac{(n-5)(n^4-5n^3+10n^2+24)}{120}$ &$\frac{(n^6-15n^5+85n^4-225n^3+274n^2-120n-360)}{360}$\\
		\hline
		$[n-1,1]$& $n-1$ & $\frac{(n-1)(n^3-5n^2+6n-24)}{24}$  & $\frac{(n-1)(n^4-9n^3+26n^2-24n-120)}{120}$ &$\frac{(n-1)(n^5-14n^4+71n^3-154n^2+120n-360)}{360}$\\
		\hline
		$[n-2,2]$& $\frac{n(n-3)}{2}$ & $\frac{(n-3)(n^3-3n^2+2n-12)}{24}$& $\frac{(n-5)(n-3)(n^3-2n^2+4n+12)}{120}$ &$\frac{(n-3)(n^5-12n^4+49n^3-78n^2+40n-180)}{360}$\\
		\hline
		$[n-2,1^2]$& $\frac{(n-1)(n-2)}{2}$ &$\frac{(n-2)(n-1)(n^2-3n-12)}{24}$ &$\frac{(n-2)(n-1)(n^3-7n^2+12n-60)}{120}$ &$\frac{(n-2)(n-1)(n^4-12n^3+47n^2-60n-180)}{360}$\\
		\hline
		$[n-3,3]$& $\frac{n(n-1)(n-5)}{6}$ &$\frac{(n-1)n(n^2-9n+26)}{24}$ & $\frac{(n-1)n(n^3-9n^2+6n+76)}{120}$&$\frac{(n-7)(n-5)(n-1)n(n^2-2n+12)}{360}$\\
		\hline
		$[n-3,2,1]$& $\frac{n(n-2)(n-4)}{3}$ &$\frac{(n-7)(n-5)(n-2)n}{24}$ &$\frac{(n-4)(n-2)n(n^2-4n-37)}{120}$ &$\frac{(n-4)(n-2)n(n^3-9n^2+23n-135)}{360}$\\
		\hline
		$[n-3,1^3]$& $\frac{(n-1)(n-2)(n-3)}{6}$ &$\frac{(n-4)(n-3)(n-2)(n-1)}{24}$ &$\frac{(n-3)(n-2)(n-1)(n^2-4n-20)}{120}$ &$\frac{(n-3)(n-2)(n-1)(n^3-9n^2+20n-60)}{360}$\\
		\hline
		$[n-4,4]$& $\frac{n(n-1)(n-2)(n-7)}{24}$ & $\frac{(n-2)(n-1)n}{6}$ &$\frac{(n-2)(n-1)n(n^2-12n+47)}{120}$ &$\frac{(n-2)(n-1)n(n+1)(n^2-13n+45)}{360}$\\
		\hline
		$[n-4,1^4]$& $\frac{(n-1)(n-2)(n-3)(n-4)}{24}$ & $\frac{(n-3)(n-2)(n-1)}{6}$ &$\frac{(n-5)(n-4)(n-3)(n-2)(n-1)}{120}$ &$\frac{(n-4)(n-3)(n-2)(n-1)(n^2-5n-15)}{360}$\\
		\hline
		$[n-4,3,1]$& $\frac{n(n-1)(n-3)(n-6)}{8}$ &  &$\frac{(n-14)(n-7)(n-3)(n-1)n}{120}$ &$\frac{(n - 3 (n - 1) n (n^3 - 11n^2 - 7n + 230)}{360}$\\
		\hline
		$[n-4,2^2]$& $\frac{n(n-1)(n-4)(n-5)}{12}$ & &$\frac{(n-8)(n-7)(n-4)(n-1)n}{120}$ &$\frac{(n-8)(n-5)(n-4)(n-1)n(n+3)}{360}$\\
		\hline
		$[n-4,2,1^2]$&$\frac{n(n-2)(n-3)(n-5)}{8}$ & & $\frac{(n-3)(n-2)n(n^2-20n+79)}{120}$&$\frac{(n-5)(n-3)(n-2)n(n^2-5n-41)}{360}$\\
		\hline
		$[n-5,5]$&$\frac{n(n-1)(n-2)(n-3)(n-9)}{120}$ & & $\frac{(n-3)(n-2)(n-1)n}{24}$&$\frac{(n-3)(n-2)(n -1)n(n^2-12n+47)}{360}$\\
		\hline
		$[n-5,1^5]$& $\frac{(n-1)(n-2)(n-3)(n-4)(n-5)}{120}$ & &$\frac{(n-4)(n-3)(n-2)(n-1)}{24}$ &$\frac{(n-5)(n-4)(n-2)(n-1)(n-3)^2}{360}$\\
		\hline
		$[n-5,4,1]$& $\frac{n(n-1)(n-2)(n-4)(n-8)}{30}$ &  & &$\frac{(n-4)(n-2)(n-1)n(n^2-20n+111)}{360}$\\
		\hline
		$[n-5,3,2]$& $\frac{n(n-1)(n-2)(n-5)(n-7)}{24}$ &  & &$\frac{(n-13)(n-9)(n-5)(n-2)(n-1)n}{360}$\\
		\hline
		$[n-5,3,1^2]$&$\frac{n(n-1)(n-3)(n-4)(n-7)}{20}$ &  & &$\frac{(n-17)(n-8)(n-4)(n-3)(n-1)n}{360}$\\
		\hline
		$[n-5,2^2,1]$& $\frac{n(n-1)(n-3)(n-5)(n-6)}{24}$ &  & &$\frac{(n-14)(n-7)(n-5)(n-3)(n-1)n}{360}$\\
		\hline
		$[n-5,2,1^3]$& $\frac{n(n-2)(n-3)(n-4)(n-6)}{30}$ &  & &$\frac{(n-11)(n-7)(n-4)(n-3)(n-2)n}{360}$\\
		\hline
		$[n-6,6]$& $\frac{n(n-1)(n-2)(n-3)(n-4)(n-11)}{720}$ &  & &$\frac{(n-4)(n-3)(n-2)(n-1)n(n+1)}{720}$\\
		\hline
		$[n-6,1^6]$& $\frac{(n-1)(n-2)(n-3)(n-4)(n-5)(n-6)}{720}$  &  & &$\frac{(n-5)(n-4)(n-3)(n-2)(n-1)(n+6)}{720}$\\
		\hline
	\end{tabular}
	\label{tab:caseee}
\end{table}
\end{landscape}

\begin{table}[ht]
	\caption{Possibility for $\phi$}
	\centering
	\begin{tabular}{|c|c|}
		\hline
		Character $\phi \downharpoonleft_n$ & $\phi$\\		
		\hline
		$[n]$& $[n+1]$, $[n,1]$ \\
		\hline
		$[n-1,1]$& $[n,1]$, $[n-1,2]$, $[n-1,1^2]$ \\
		\hline
		$[n-2,2]$& $[n-1,2]$,$[n-2,3]$, $[n-2,2,1]$ \\
		\hline
		$[n-2,1^2]$&$[n-1,1^2]$,$[n-2,2,1]$,$[n-2,1^3]$ \\
		\hline
		$[n-3,3]$&$[n-2,3]$, $[n-3,4]$, $[n-3,3,1]$ \\
		\hline
		$[n-3,2,1]$&$[n-2,2,1]$, $[n-3,3,1]$, $[n-3,2^2]$, $[n-3,2,1^2]$ \\
		\hline
		$[n-3,1^3]$& $[n-2,1^3]$,$[n-3,2,1^2]$,$[n-3,1^4]$ \\
		\hline
		$[n-4,4]$&$[n-3,4]$,$[n-4,5]$,$[n-4,4,1]$ \\
		\hline
		$[n-4,1^4]$&$[n-3,1^4]$,$[n-4,2,1^3]$,$[n-4,1^5]$ \\
		\hline
		$[n-4,3,1]$&$[n-3,3,1]$, $[n-4,4,1]$,$[n-4,3,2]$,$[n-4,3,1^2]$ \\
		\hline
		$[n-4,2^2]$&$[n-3,2^2]$,$[n-4,3,2]$, $[n-4,2^2,1]$ \\
		\hline
		$[n-4,2,1^2]$&$[n-3,2,1^2]$, $[n-4,3,1^2]$,$[n-4,2^2,1]$,$[n-4,2,1^3]$ \\
		\hline
		$[n-5,5]$&$[n-4,5]$, $[n-5,6]$, $[n-5,5,1]$ \\
		\hline
		$[n-5,1^5]$&$[n-4,1^5]$, $[n-5,2,1^4]$, $[n-5,1^6]$ \\
		\hline
		$[n-5,4,1]$&$[n-4,4,1]$,$[n-5,5,1]$,$[n-5,4,2]$,$[n-5,4,1^2]$\\
		\hline
		$[n-5,3,2]$&$[n-4,3,2]$,$[n-5,4,2]$,$[n-5,3^2]$,$[n-5,3,2,1]$ \\
		\hline
		$[n-5,3,1^2]$&$[n-4,3,1^2]$,$[n-5,4,1^2]$,$[n-5,3,2,1]$,$[n-5,3,1^3]$ \\
		\hline
		$[n-5,2^2,1]$&$[n-4,2^2,1]$,$[n-5,3,2,1]$,$[n-5,2^3]$,$[n-5,2^2,1^2]$ \\
		\hline
		$[n-5,2,1^3]$&$[n-4,2,1^3]$, $[n-5,3,1^3]$, $[n-5,2^2,1^2]$, $[n-5,2,1^4]$ \\
		\hline
		$[n-6,6]$&$[n-5,6]$, $[n-6,7]$,$[n-6,6,1]$ \\
		\hline
		$[n-6,1^6]$& $[n-5,1^6]$,$[n-6,2,1^5]$, $[n-6,1^7]$ \\
		\hline
	\end{tabular}
	\label{tab:poss}
\end{table}
\vskip 1cm
\noindent
\begin{table}[ht]
	\caption{Exception for $k=4$ and $n \geq 15$}
	\centering
	\begin{tabular}{|c|c|c|}
		\hline
		Character & Degree $D$& $D-\binom{n}{4}$\\
		\hline 
		$[n-3,3,1]$ & $\frac{(n+1)n(n-2)(n-5)}{8}$&$\frac{(n-7)(n-2)n(n+1)}{12}>0$\\
		\hline
		$[n-3,2^2]$&$\frac{(n+1)n(n-3)(n-4)}{12}$&$\frac{n(n+1)(n^2-11n+22)}{24}>0$\\
		\hline
		$[n-3,2,1^2]$&$\frac{(n+1)(n-1)(n-2)(n-4)}{8}$&$\frac{(n-6)(n-2)(n-1)(n+1)}{12}>0$\\
		\hline
		$[n-4,5]$&$\frac{(n+1)n(n-1)(n-2)(n-8)}{120}$&$\frac{(n-13)(n-2)(n-1)n(n+1)}{120}>0$\\
		\hline
		$[n-4,4,1]$&$\frac{(n+1)n(n-1)(n-3)(n-7)}{30}$&$\frac{(n-1)n(n+1)(n^2-45n/4+47/2)}{30}>0$\\
		\hline
		$[n-4,2,1^3]$&$\frac{(n+1)(n-1)(n-2)(n-3)(n-5)}{30}$&$\frac{(n-2)(n-1)(n+1)(n^2-37n/4+15)}{30}>0$\\
		\hline
		$[n-4,1^5]$&$\frac{n(n-1)(n-2)(n-3)(n-4)}{120}$&$\frac{(n-2)(n-1)n(n^2-12n+7)}{120}>0$\\
		\hline
	\end{tabular}
	\label{tab:k4}
\end{table}
\vskip 1cm
\noindent
\begin{table}[ht]
	\caption{Exception for $k=5$ and $n \geq 19$}
	\centering
	\begin{tabular}{|c|c|c|}
		\hline
		Character & Degree $D$& $D-\binom{n}{5}$\\
		\hline
		$[n-4,4,1]$&$\frac{(n+1)n(n-1)(n-3)(n-7)}{30}$&$\frac{(n-26/3)(n-3)(n-1)n(n+1)}{40}>0$\\
		\hline
		$[n-4,2,1^3]$&$\frac{(n+1)(n-1)(n-2)(n-3)(n-5)}{30}$&$\frac{(n-20/3)(n-3)(n-2)(n-1)(n+1)}{40}>0$\\
		\hline
		$[n-4,3,2]$&$\frac{(n+1)n(n-1)(n-4)(n-6)}{24}$&$\frac{(n-1)n(n+1)(n^2-45n/4+57/2)}{30}>0$\\
		\hline
		$[n-4,3,1^2]$&$\frac{(n+1)n(n-2)(n-3)(n-6)}{20}$&$\frac{(n-7)(n-3)(n-2)n(n+1)}{24}>0$\\
		\hline
		$[n-4,2^2,1]$&$\frac{(n+1)n(n-2)(n-4)(n-5)}{24}$&$\frac{(n-2)n(n+1)(n^2-41n/4+97/4)}{30}>0$\\
		\hline
		$[n-5,6]$&$\frac{(n+1)n(n-1)(n-2)(n-3)(n-10)}{720}$&$\frac{(n-16)(n-3)(n-2)(n-1)n(n+1)}{720}>0$\\
		\hline
		$[n-5,5,1]$&$\frac{(n+1)n(n-1)(n-2)(n-4)(n-9)}{144}$&$\frac{(n-2)(n-1)n(n+1)(n^2-71n/5+198/5)}{144}>0$\\
		\hline
		$[n-5,2,1^4]$&$\frac{(n+1)(n-1)(n-2)(n-3)(n-4)(n-6)}{144}$&$\frac{(n-3)(n-2)(n-1)(n+1)(n^2-56n/5+24)}{144}>0$\\
		\hline
		$[n-5,1^6]$&$\frac{n(n-1)(n-2)(n-3)(n-4)(n-5)}{720}$&$\frac{(n-14)(n-3)(n-2)n(n-1)^2}{720}>0$\\
		\hline
	\end{tabular}
	\label{tab:k5}
\end{table}
\vskip 1cm
\noindent
\begin{table}[ht]
	\caption{Exception for $k=6$ and $n \geq 27$}
	\centering
	\begin{tabular}{|c|c|c|}
		\hline
		Character & Degree $D$ & $D-2\binom{n}{6}$ \\
		\hline
		$[n-5,5,1]$&$\frac{(n+1)n(n-1)(n-2)(n-4)(n-9)}{144}$&$\frac{(n-13)(n-4)(n-2)(n-1)n(n+1)}{240}>0$\\
		\hline
		$[n-5,2,1^4]$&$\frac{(n+1)(n-1)(n-2)(n-3)(n-4)(n-6)}{144}$&$\frac{(n-10)(n-4)(n-3)(n-2)(n-1)(n+1)}{240}>0$\\
		\hline
		$[n-5,4,2]$&$\frac{(n+1)n(n-1)(n-2)(n-5)(n-8)}{80}$&$\frac{7(n-2)(n-1)n(n+1)(n^2-103n/7+48)}{720}>0$\\
		\hline
		$[n-5,4,1^2]$&$\frac{(n+1)n(n-1)(n-3)(n-4)(n-8)}{72}$&$\frac{(n-19/2)(n-4)(n-3)(n-1)n(n+1)}{90}>0$\\
		\hline
		$[n-5,3^2]$&$\frac{(n+1)n(n-1)(n-2)(n-6)(n-7)}{144}$&$\frac{(n-2)(n-1)n(n+1)(n^2-17n+62)}{240}>0$\\
		\hline
		$[n-5,3,2,1]$&$\frac{(n+1)n(n-1)(n-3)(n-5)(n-7)}{45}$&$\frac{7(n-8)(n-34/7)(n-3)(n-1)n(n+1)}{360}>0$\\
		\hline
		$[n-5,3,1^3]$&$\frac{(n+1)n(n-2)(n-3)(n-4)(n-7)}{72}$&$\frac{(n-17/2)(n-4)(n-3)(n-2)n(n+1)}{90}>0$\\
		\hline
		$[n-5,2^3]$&$\frac{(n+1)n(n-1)(n-4)(n-5)(n-6)}{144}$&$\frac{(n-4)(n-1)n(n+1)(n^2-15n+46)}{240}>0$\\
		\hline
		$[n-5,2^2,1^2]$&$\frac{(n+1)n(n-2)(n-3)(n-5)(n-6)}{80}$&$\frac{7(n-3)(n-2)n(n+1)(n^2-89n/7+262/7)}{720}>0$\\
		\hline
		$[n-6,7]$&$\frac{(n+1)n(n-1)(n-2)(n-3)(n-4)(n-12)}{5040}$&$\frac{(n-26)(n-4)(n-3)(n-2)(n-1)n(n+1)}{5040}>0$\\
		\hline
		$[n-6,6,1]$&$\frac{(n+1)n(n-1)(n-2)(n-3)(n-5)(n-11)}{840}$&$\frac{(n-3)(n-2)(n-1)n(n+1)(n^2-55n/3+193/3)}{840}>0$\\
		\hline
		$[n-6,2,1^5]$&$\frac{(n+1)(n-1)(n-2)(n-3)(n-4)(n-5)(n-7)}{840}$&$\frac{(n-4)(n-3)(n-2)(n-1)(n+1)(n^2-43n/3+35)}{840}>0$\\
		\hline
		$[n-6,1^7]$&$\frac{n(n-1)(n-2)(n-3)(n-4)(n-5)(n-6)}{5040}$&$\frac{(n-4)(n-3)(n-2)(n-1)n(n^2-25n+16)}{5040}>0$\\
		\hline
	\end{tabular}
	\label{tab:k6}
\end{table}

\end{document}